\documentclass[a4paper,reqno]{article}

\usepackage[utf8]{inputenc}
\usepackage[T1]{fontenc}

\usepackage{fullpage}
\usepackage{booktabs}
\usepackage{paralist}

\usepackage{comment}

\newcommand{\nquad}{\!\!\!\!\!\!}

\usepackage{bm}

\usepackage{amsmath, amsthm, amssymb,mathtools}

\usepackage{tikz}
\usetikzlibrary{decorations.pathreplacing}
\usetikzlibrary{decorations.pathmorphing}
\usetikzlibrary{decorations.markings}
\usetikzlibrary{calc,arrows}
\tikzstyle arrowstyle=[scale=1]
\tikzstyle directed=[postaction={decorate,decoration={markings,
		mark=at position 0.65 with {\arrow[arrowstyle]{stealth}}}}]

\usepackage[outline]{contour}
\contourlength{2pt}

\usepackage[centertableaux]{ytableau}

\colorlet{mycolor}{blue}

\newtheorem{theorem}{Theorem}

\newtheorem{proposition}[theorem]{Proposition}
\newtheorem{problem}[theorem]{Problem}

\theoremstyle{definition}
\newtheorem{definition}[theorem]{Definition}

\numberwithin{equation}{section}

\DeclareMathOperator{\ASTZ}{\mathcal{AST\!Z}}

\DeclareMathOperator{\CSSPP}{\mathcal{CSSPP}}
\DeclareMathOperator{\NILP}{NILP}
\DeclareMathOperator{\inv}{inv}

\usepackage[bookmarks,hyperindex,colorlinks]{hyperref}

\begin{document}
	
\title{Weight-Preserving Bijections Between Integer Partitions and a Class of Alternating Sign Trapezoids}
\author{Hans H{\"o}ngesberg\thanks{The author acknowledges support from the Austrian Science Foundation FWF, SFB grant F50 and grant P34931.}}
\date{}

\maketitle

\begin{abstract}
	We construct weight-preserving bijections between column strict shifted plane partitions with one row and alternating sign trapezoids with exactly one column in the left half that sums to $1$. Amongst other things, they relate the number of $-1$s in the alternating sign trapezoids to certain elements in the column strict shifted plane partitions that generalise the notion of special parts in descending plane partitions. The advantage of these bijections is that they include configurations with $-1$s, which is a feature that many of the bijections in the realm of alternating sign arrays lack.
\end{abstract}

\section{Introduction}

Alternating sign matrices (ASMs) are square matrices with entries $-1$, $0$, and $1$ such that the nonzero entries alternate in sign and sum to $1$ along each row and column. When Mills, Robbins, and Rumsey conjectured that $n \times n$ ASMs are equinumerous with totally symmetric self-complementary plane partitions (TSSCPPs) in an $2n \times 2n \times 2n$ box \cite{MRR86} and with descending plane partitions (DPPs) without parts exceeding $n$ \cite{MRR83}, they initiated a strenuous quest for weight-preserving bijections between ASMs and other presumably equinumerous families of objects. These conjectures were proved nonbijectively by establishing the same enumeration formula for each of the three classes: for DPPs \cite{And79}, TSSCPPs \cite{And94}, and ASMs \cite{Zei96}, \cite{Kup96}. A fourth equinumerous class was introduced in a recent work by Ayyer, Behrend, and Fischer \cite{ABF20}, namely alternating sign triangles (ASTs). 

The realm of alternating sign arrays has been known for the lack of bijective proofs. Most of the bijections that have been established only consider special cases. For instance, there are partial bijections between ASMs and TSSCPPs by Cheballah and Biane \cite{CB12} and Striker \cite{Str18} and between ASMs and DPPs by Ayyer \cite{Ayy10}, Striker \cite{Str11}, and Fulmek \cite{Ful20}.

It took around four decades to find the first general bijective proof. Fischer and Konvalinka \cite{FK20a}, \cite{FK20b} have recently constructed an explicit but rather intricate bijection that relates ASMs with DPPs. Its underlying concept is turning Fischer's nonbijective proof of the ASM enumeration formula \cite{Fis16} into a bijective proof by using a generalisation of the Garsia--Milne involution principle. See \cite{FK20c} for a concise overview of how the proof was found and how the bijections work.

Ayyer, Behrend, and Fischer \cite{ABF20} announced an infinite family of alternating sign arrays that are generalising ASTs: alternating sign trapezoids (ASTZs). Fischer \cite{Fis19} proved that ASTZs are equinumerous with column strict shifted plane partitions (CSSPPs) of a fixed class which are a simple generalisation of DPPs essentially introduced by Andrews \cite{And79}. CSSPPs can also be interpreted as cyclically symmetric lozenge tilings of a hexagon with a central triangular hole \cite{BF}. This correspondence is also explicitly mentioned in \cite[Section~4]{Kra06} and implicitly present in \cite[Section~3]{CK00}.

In this paper, we present weight-preserving bijections between CSSPPs of a fixed class with one row and ASTZs with exactly one column in the left half that sums to $1$. We start with the definitions of ASTZs and CSSPPs and of the statistics we consider on these objects.

\begin{definition}
	For $n \geq 1$ and $l \geq 2$, an \emph{$(n,l)$-alternating sign trapezoid} is an array $(a_{i,j})_{1 \le i \le n,i \le j \le 2n+l-1-i}$ of $-1$s, $0$s, and $+1$s in a trapezoidal shape with $n$ rows
	of the following form
	\begin{equation*}
	\begin{array}[t]{ccccccccc}
	a_{1,1}&a_{1,2}&\cdots&\cdots&\cdots&\cdots&\cdots&\cdots&a_{1,2n+l-2}\\
	&a_{2,2}&\cdots&\cdots&\cdots&\cdots&\cdots&a_{2,2n+l-3}&\\
	&&\ddots&&&&\reflectbox{$\ddots$}&&\\
	&&&a_{n,n}&\cdots&a_{n,n+l-1}&&&
	\end{array}
	\end{equation*}
	such that the following four conditions hold:
	\begin{compactitem}
		\item the nonzero entries alternate in sign in each row and each column;
		\item the topmost nonzero entry in each column is $1$;
		\item each row sums to $1$; 
		\item each of the central $l-2$ columns sums to $0$.
	\end{compactitem}
	
	If a column sums to $1$, we call it a \emph{$1$-column}. Otherwise, it is a \emph{$0$-column}. In addition, if the bottom entry of a $1$-column is $0$, we also call it a \emph{$10$-column}; similarly, we call a $1$-column a \emph{$11$-column} if the bottom entry is $1$.
\end{definition}

\begin{figure}[htb]
	\centering
	\begin{equation*}
	\begin{array}{ccccccccccccc}
	0 & 0 & 0 & 0 & 0 & 0 & 1 & 0 & 0 & 0 & 0 & 0 & 0 \\
	& 0 & 1 & 0 & 0 & 0 & -1 & 0 & 1 & 0 & 0 & 0 & \\
	& & 0 & 0 & 0 & 0 & 1 & 0 & 0 & 0 & 0 & & \\
	& & & 1 & 0 & 0 & 0 & 0 & -1 & 1 & & & \\
	& & & & 1 & 0 & -1 & 0 & 1 & & & &
	\end{array}
	\end{equation*}
	
	\caption{Example of a $(5,5)$-ASTZ.}
	\label{fig:exampleASTZ}
\end{figure}

See Figure~\ref{fig:exampleASTZ} for an example of a $(5,5)$-ASTZ. Note that ASTZs can be interpreted as a generalisation of ASTs as follows: We can construct the class of ASTs with $n+1$ rows by adding an additional bottom row that consists of a single $1$ below $(n,3)$-ASTZs such that we obtain a triangular array. Ayyer, Behrend, and Fischer \cite{ABF20} showed that the number of $-1$s is equally distributed in ASTs with $n$ rows and $n \times n$ ASMs.

Let $\ASTZ_{n,l}$ denote the set of $(n,l)$-ASTZs. The case $l=1$ is special and we shall treat it separately in Section~\ref{sec:QAST}. In the following, we assume $l \ge 2$ until further notice. We introduce four different statistics on ASTZs. For $A \in \ASTZ_{n,l}$, we define
\begin{align*}
\mu(A) &\coloneqq \text{\# $-1$s in A,}\\
r(A) &\coloneqq \text{\# $1$-columns among the $n$ leftmost columns of A,}\\
p(A) &\coloneqq \text{\# $10$-columns among the $n$ leftmost columns of A,}\\
q(A) &\coloneqq \text{\# $10$-columns among the $n$ rightmost columns of A.}
\end{align*}
Finally, the weight of $A$ is set to be
\begin{equation*}
M^{\mu(A)} R^{r(A)} P^{p(A)} Q^{q(A)}.
\end{equation*}
Table~\ref{tab:(2,4)-ASTZs3Enum} shows the elements of $\ASTZ_{2,4}$ together with the corresponding weights.

\begin{table}[ht]
	\centering
	\caption{The eight ASTZs in $\ASTZ_{2,4}$ together with their weights.}
	\begin{math}
		\begin{array}{cccc}
			\toprule
			\begin{array}[t]{cccccc}
				1  &  0  &  0  &  0  &  0  &  0 \\
				&  1  &  0  &  0  &  0  &
			\end{array}
			
			&
			
			\begin{array}[t]{cccccc}
				0  &  0  &  0  &  0  &  1  &  0 \\
				&  1  &  0  &  0  &  0  &
			\end{array}
			
			&
			
			\begin{array}[t]{cccccc}
				0  &  0  &  0  &  0  &  0  &  1 \\
				&  1  &  0  &  0  &  0  &
			\end{array}
			
			&
			
			\begin{array}[t]{cccccc}
				1  &  0  &  0  &  0  &  0  &  0 \\
				&  0  &  0  &  0  &  1  &
			\end{array}
			
			\\\midrule
			
			R^2 & R Q & R & R
			
			\\\midrule\midrule
			
			\begin{array}[t]{cccccc}
				0  &  1  &  0  &  0  &  0  &  0 \\
				&  0  &  0  &  0  &  1  &
			\end{array}
			
			&
			
			\begin{array}[t]{cccccc}
				0  &  0  &  0  &  0  &  0  &  1 \\
				&  0  &  0  &  0  &  1  &
			\end{array}
			
			&
			
			\begin{array}[t]{cccccc}
				0  &  0  &  1  &  0  &  0  &  0 \\
				&  1  &  -1  &  0  &  1  &
			\end{array}
			
			&
			
			\begin{array}[t]{cccccc}
				0  &  0  &  0  &  1  &  0  &  0 \\
				&  1  &  0  &  -1  &  1  &
			\end{array}
			
			\\\midrule
			
			R P & 1 & M R & M R
			
			\\\bottomrule
		\end{array}
	\end{math}
	\label{tab:(2,4)-ASTZs3Enum}
\end{table}

\begin{definition}	
	A \emph{strict partition} $\mu$ is a tuple of strictly decreasing positive integers $\mu_i$; the number of elements in $\mu$ is denoted by $\ell(\mu)$. A \emph{shifted Young diagram} of \emph{shape} $\mu$ is a finite collection of cells arranged in $\ell(\mu)$ rows such that row~$i$ has length $\mu_i$, and each row is indented by one cell to the right compared to the row above.
	
	A \emph{column strict shifted plane partition} $\pi=(\pi_{i,j})_{1 \le i \le \ell(\mu),i \le j \le i + \mu_i -1}$ is a filling of a shifted Young diagram of shape $\mu$ with positive integers such that the entries weakly decrease along each row and strictly decrease down each column. We call the entries \emph{parts} and say that the partition is of \emph{class~k} if the first part of each row equals $k$ plus its corresponding row length; that is, $\pi_{i,i}=k+\mu_i$ for all $1 \le i \le \ell(\mu)$.
\end{definition}
	
\begin{figure}[ht]
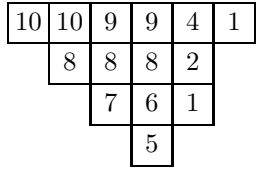

	\centering
	\begin{ytableau}
		10 & 10 & 9 & 9 & 4 & 1\\
		\none & 8 & 8 & 8 & 2\\
		\none & \none & 7 & 6 & 1\\
		\none & \none & \none & 5\\
	\end{ytableau}
	\caption{Example of a CSSPP of shape $(6,4,3,1)$ and of class $4$.}
	\label{fig:exampleCSSPP}
\end{figure}

	A CSSPP $\pi$ of class~$k$ can be interpreted as a cyclically symmetric lozenge tiling of a hexagon with two distinct side lengths $\pi_{1,1}$ and $\pi_{1,1}-k$ and with a central triangular hole of size~$k$. See Figure~\ref{fig:exampleCSSPP} for an example of a CSSPP. Figure~\ref{fig:exampleRhombusTiling} illustrates the corresponding interpretation as a lozenge tiling. Note that we consider the collection of zero cells to be a CSSPP of class $k$ for any $k$.
	
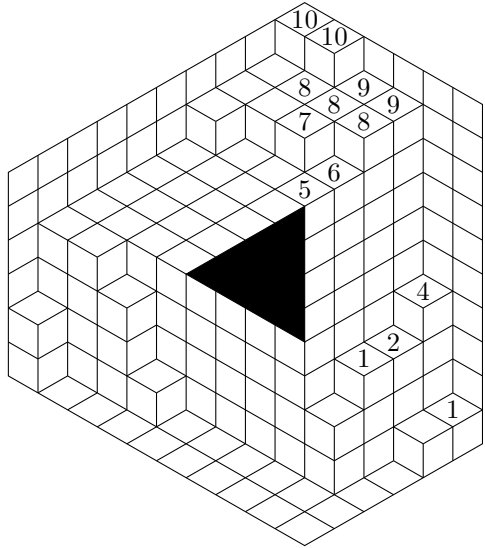
\begin{figure}[ht]
	\centering
	\begin{tikzpicture}[scale=.45]
	\clip ($ (30:-1) + (0,-5) $) rectangle ($ (30:17) + (0,3) $);
	
	\foreach \x in {0,...,16}
	\foreach \y in {-10,...,6}
	\node (\x/\y) at ($ (30:\x) + (0,\y) $) {};
	
	\node at ($ (30:10) + (0,5.5) $) {$10$};
	\node at ($ (30:11) + (0,4.5) $) {$10$};
	\node at ($ (30:12) + (0,2.5) $) {$9$};
	\node at ($ (30:13) + (0,1.5) $) {$9$};
	\node at ($ (30:14) + (0,-4.5) $) {$4$};
	\node at ($ (30:15) + (0,-8.5) $) {$1$};
	
	\node at ($ (30:10) + (0,3.5) $) {$8$};
	\node at ($ (30:11) + (0,2.5) $) {$8$};
	\node at ($ (30:12) + (0,1.5) $) {$8$};
	\node at ($ (30:13) + (0,-5.5) $) {$2$};
	
	\node at ($ (30:10) + (0,2.5) $) {$7$};
	\node at ($ (30:11) + (0,0.5) $) {$6$};
	\node at ($ (30:12) + (0,-5.5) $) {$1$};
	
	\node at ($ (30:10) + (0,0.5) $) {$5$};
	
	\draw (0/0.center) -- (0/6.center) -- (10/6.center) -- (16/0.center) -- (16/-10.center) -- (10/-10.center) -- cycle;
	
	\draw[fill] (6/0.center) -- (10/0.center) -- (10/-4.center) -- cycle;
	
	\draw (0/5.center) -- (9/5.center);
	\draw (10/5.center) -- (11/5.center);
	
	\draw (0/4.center) -- (6/4.center);
	\draw (7/4.center) -- (10/4.center);
	\draw (11/4.center) -- (12/4.center);
	
	\draw (0/3.center) -- (1/3.center);
	\draw (2/3.center) -- (8/3.center);
	\draw (9/3.center) -- (12/3.center);
	
	\draw (0/2.center) -- (1/2.center);
	\draw (3/2.center) -- (9/2.center);
	\draw (10/2.center) -- (13/2.center);
	
	\draw (1/1.center) -- (2/1.center);
	\draw (3/1.center) -- (4/1.center);
	\draw (5/1.center) -- (11/1.center);
	\draw (12/1.center) -- (14/1.center);
	
	\draw (1/0.center) -- (2/0.center);
	\draw (4/0.center) -- (5/0.center);
	\draw (10/0.center) -- (14/0.center);
	
	\draw (1/-1.center) -- (2/-1.center);
	\draw (4/-1.center) -- (5/-1.center);
	\draw (10/-1.center) -- (14/-1.center);
	
	\draw (2/-2.center) -- (3/-2.center);
	\draw (4/-2.center) -- (5/-2.center);
	\draw (10/-2.center) -- (14/-2.center);
	
	\draw (3/-3.center) -- (4/-3.center);
	\draw (5/-3.center) -- (6/-3.center);
	\draw (10/-3.center) -- (14/-3.center);
	
	\draw (4/-4.center) -- (6/-4.center);
	\draw (10/-4.center) -- (14/-4.center);
	
	\draw (5/-5.center) -- (7/-5.center);
	\draw (10/-5.center) -- (13/-5.center);
	\draw (14/-5.center) -- (15/-5.center);
	
	\draw (6/-6.center) -- (8/-6.center);
	\draw (10/-6.center) -- (11/-6.center);
	\draw (12/-6.center) -- (15/-6.center);
	
	\draw (7/-7.center) -- (9/-7.center);
	\draw (11/-7.center) -- (15/-7.center);
	
	\draw (8/-8.center) -- (10/-8.center);
	\draw (11/-8.center) -- (15/-8.center);
	
	\draw (9/-9.center) -- (13/-9.center);
	\draw (14/-9.center) -- (16/-9.center);
	
	\draw (15/1.center) -- (15/-8.center);
	\draw (15/-9.center) -- (15/-10.center);
	
	\draw (14/2.center) -- (14/-4.center);
	\draw (14/-5.center) -- (14/-8.center);
	\draw (14/-9.center) -- (14/-10.center);
	
	\draw (13/3.center) -- (13/2.center);
	\draw (13/1.center) -- (13/-5.center);
	\draw (13/-6.center) -- (13/-9.center);
	
	\draw (12/4.center) -- (12/3.center);
	\draw (12/1.center) -- (12/-5.center);
	\draw (12/-6.center) -- (12/-9.center);
	
	\draw (11/4.center) -- (11/3.center);
	\draw (11/2.center) -- (11/1.center);
	\draw (11/0.center) -- (11/-6.center);
	\draw (11/-7.center) -- (11/-9.center);
	
	\draw (10/5.center) -- (10/4.center);
	\draw (10/2.center) -- (10/1.center);
	\draw (10/-4.center) -- (10/-8.center);
	
	\draw (9/6.center) -- (9/5.center);
	\draw (9/3.center) -- (9/2.center);
	\draw (9/-3.center) -- (9/-7.center);
	
	\draw (8/6.center) -- (8/5.center);
	\draw (8/4.center) -- (8/3.center);
	\draw (8/-2.center) -- (8/-6.center);
	
	\draw (7/6.center) -- (7/5.center);
	\draw (7/4.center) -- (7/3.center);
	\draw (7/-1.center) -- (7/-5.center);
	
	\draw (6/6.center) -- (6/4.center);
	\draw (6/0.center) -- (6/-4.center);
	
	\draw (5/6.center) -- (5/4.center);
	\draw (5/1.center) -- (5/-2.center);
	\draw (5/-3.center) -- (5/-4.center);
	
	\draw (4/6.center) -- (4/4.center);
	\draw (4/2.center) -- (4/1.center);
	\draw (4/0.center) -- (4/-3.center);
	
	\draw (3/6.center) -- (3/4.center);
	\draw (3/2.center) -- (3/-2.center);
	
	\draw (2/6.center) -- (2/4.center);
	\draw (2/3.center) -- (2/-1.center);
	
	\draw (1/6.center) -- (1/2.center);
	\draw (1/1.center) -- (1/-1.center);
	
	\draw (11/-10.center) -- (2/-1.center);
	\draw (1/0.center) -- (0/1.center);
	
	\draw (12/-10.center) -- (6/-4.center);
	\draw (5/-3.center) -- (2/0.center);
	\draw (1/1.center) -- (0/2.center);
	
	\draw (13/-10.center) -- (12/-9.center);
	\draw (11/-8.center) -- (5/-2.center);
	\draw (4/-1.center) -- (1/2.center);
	
	\draw (14/-10.center) -- (13/-9.center);
	\draw (11/-7.center) -- (5/-1.center);
	\draw (4/0.center) -- (1/3.center);
	
	\draw (14/-9.center) -- (13/-8.center);
	\draw (12/-7.center) -- (11/-6.center);
	\draw (10/-5.center) -- (4/1.center);
	\draw (3/2.center) -- (1/4.center);
	
	\draw (15/-9.center) -- (14/-8.center);
	\draw (12/-6.center) -- (11/-5.center);
	\draw (6/0.center) -- (2/4.center);
	
	\draw (16/-9.center) -- (15/-8.center);
	\draw (13/-6.center) -- (12/-5.center);
	\draw (7/0.center) -- (3/4.center);
	
	\draw (16/-8.center) -- (15/-7.center);
	\draw (14/-6.center) -- (13/-5.center);
	\draw (8/0.center) -- (4/4.center);
	
	\draw (16/-7.center) -- (15/-6.center);
	\draw (14/-5.center) -- (13/-4.center);
	\draw (9/0.center) -- (5/4.center);
	
	\draw (16/-6.center) -- (14/-4.center);
	\draw (10/0.center) -- (6/4.center);
	
	\draw (16/-5.center) -- (14/-3.center);
	\draw (11/0.center) -- (8/3.center);
	\draw (7/4.center) -- (6/5.center);
	
	\draw (16/-4.center) -- (14/-2.center);
	\draw (12/0.center) -- (11/1.center);
	\draw (10/2.center) -- (7/5.center);
	
	\draw (16/-3.center) -- (14/-1.center);
	\draw (12/1.center) -- (8/5.center);
	
	\draw (16/-2.center) -- (14/0.center);
	\draw (13/1.center) -- (9/5.center);
	
	\draw (16/-1.center) -- (12/3.center);
	\draw (11/4.center) -- (9/6.center);
	\end{tikzpicture}
	\caption{The cyclically symmetric lozenge tiling of a hexagon with a central triangular hole that corresponds to the CSSPP in Figure~\ref{fig:exampleCSSPP}.}
	\label{fig:exampleRhombusTiling}
\end{figure}
	
	Let $\CSSPP_{n,k}$ denote the set of CSSPPs of class $k$ with at most $n$ parts in the first row. We introduce four different statistics of which two depend on a fixed parameter~$d\in \{1,\dots,k\}$. For $\pi \in \CSSPP_{n,k}$, we define
	\begin{align*}
	\mu_d(\pi) &\coloneqq \text{\# parts $\pi_{i,j}\in\{2,3,\dots,j-i+k\} \setminus \{j-i+d\}$,}\\
	r(\pi) &\coloneqq \text{\# rows of $\pi$,}\\
	p_d(\pi) &\coloneqq \text{\# parts $\pi_{i,j}=j-i+d$,}\\
	q(\pi) &\coloneqq \text{\# $1$s in $\pi$.}
	\end{align*}
The weight of $\pi$ is set to be
\begin{equation*}
M^{\mu_d(\pi)} R^{r(\pi)} P^{p_d(\pi)} Q^{q(\pi)}.
\end{equation*}
Table~\ref{tab:23CSSPPs3Enum} shows the elements of $\CSSPP_{2,3}$ together with the corresponding weights for all $1 \leq d \leq 3$.

\begin{table}[ht]
	\centering
	\caption{The eight CSSPPs in $\CSSPP_{2,3}$ together with their weights for all $1 \leq d \leq 3$.}
	\begin{math}
		\begin{array}{lcccccccc}
			\toprule
			
			&
			
			\begin{array}{c}
				\emptyset
			\end{array}
			
			&

			\ytableaushort{4}
			
			&
			
			\ytableaushort{55}
			
			&
			
			\ytableaushort{54}
			
			&
			
			\ytableaushort{53}
			
			&
			
			\ytableaushort{52}
			
			&
			
			\ytableaushort{51}
			
			&
			
			\ytableaushort{55,\none 4}
			
			\\\midrule

			d=1: & 1 & R & R & M R & M R & R P & R Q & R^2
			
			\\
			
			d=2: & 1 & R & R & M R & R P & M R & R Q & R^2
			
			\\
			
			d=3: & 1 & R & R & R P & M R & M R & R Q & R^2
			
			\\\bottomrule
		\end{array}
	\end{math}
	\label{tab:23CSSPPs3Enum}
\end{table}
	
Note that CSSPPs of class~$2$ correspond to DPPs by subtracting $1$ from each part of the CSSPP and deleting all parts equal to $0$ afterwards.
We call the parts $\pi_{i,j}\in\{2,3,\dots,j-i+k\} \setminus \{j-i+d\}$, that are counted by the statistic~$\mu_d$, \emph{$d$-special parts}. They generalise the \emph{special parts} defined by Mills, Robbins, and Rumsey \cite{MRR83} since the number of special parts in $\pi \in \CSSPP_{n,2}$ equals $\mu_2(\pi)$.
 
The author \cite{Hon} showed that the joint distribution of the respective four statistics on the sets $\ASTZ_{n,l}$ and $\CSSPP_{n,l-1}$ coincide. Their generating functions are each given by
\begin{equation}
\label{eq:GenFunc}
\det_{0 \le i,j \le n-1} \left(R \sum_{k = 0}^{i} Q^{i-k} \sum_{m=0}^{j} \binom{j}{m} M^{k-m} \left( \binom{k+l-3}{k-m} + \binom{k+l-3}{k-m-1} P M^{-1} \right) + \delta_{i,j} \right).
\end{equation}
For $n=2$ and $l=4$, \eqref{eq:GenFunc} yields $1+2R+2MR+RP+RQ+R^2$, which matches with the examples given in Tables~\ref{tab:(2,4)-ASTZs3Enum} and \ref{tab:23CSSPPs3Enum}.

In this paper, we construct explicit weight-preserving bijections between the sets $\{A \in \ASTZ_{n,l} \mid r(A) = 1\}$ and $\{\pi \in \CSSPP_{n,l-1} \mid r(\pi) = 1\}$ for all $1 \leq d \leq l-1$. To this end, we fix some notation: We number -- from left to right -- the $n$ leftmost columns of $A\in\ASTZ_{n,l}$ from $-n$ to $-1$ and the $n$ rightmost columns from $1$ to $n$. We observe that $A$ has exactly $n$ $1$-columns. If $r(A)=1$, that is, $A$ has exactly one $1$-column among the $n$ leftmost columns, then $A$ has also exactly one $0$-column among the $n$ rightmost columns. We denote the set of $(n,l)$-ASTZs with a unique $1$-column among the $n$ leftmost columns at position~$-i$ and a unique $0$-column among the $n$ rightmost columns at position~$j$ by $\ASTZ_{n,l}^{i,j}$. See Figure~\ref{fig:RunningExampleASTZ} for an example. We shall use it as a running example throughout this paper and call it $\bm{A}$.

\begin{figure}[htb]
	\centering
	\begin{equation*}
	\begin{array}{cccccccccccccccccccc}
	0 & 0 & 0 & 0 & 0 & 0 & 0 & 0 & 0 & 0 & 0 & 0 & 0 & 0 & 0 & 0 & 0 & 0 & 0 & 1 \\	
	& 0 & 0 & 0 & 0 & 0 & 0 & 0 & 0 & 0 & 0 & 0 & 0 & 0 & 0 & 0 & 0 & 1 & 0 & \\
	& & 0 & 0 & 0 & 0 & 0 & 0 & 0 & 0 & 0 & 0 & 0 & 0 & 0 & 1 & 0 & 0 & & \\
	& & & 0 & 0 & 0 & 0 & 0 & 0 & 0 & 0 & 0 & 0 & 0 & 1 & -1 & 1 & & & \\
	& & & & 0 & 0 & 0 & 0 & 0 & 0 & 0 & 0 & 0 & 0 & 0 & 1 & & & & \\
	& & & & & 0 & 0 & 0 & 0 & 0 & 1 & 0 & 0 & 0 & 0 & & & & & \\
	& & & & & & 0 & 1 & 0 & 0 & -1 & 0 & 0 & 1 & & & & & & \\
	& & & & & & & 0 & 0 & 0 & 0 & 0 & 1 & & & & & & & \\
	& & & & & & & & 0 & 0 & 0 & 1 & & & & & & & & 
	\end{array}
	\end{equation*}
	
	\caption{Our running example $\bm{A}\in\ASTZ_{9,4}^{2,8}$.}
	\label{fig:RunningExampleASTZ}
\end{figure}

Before we start with the construction of the bijections, we state some simple facts concerning the generating function $Z (n,l;i,j)$ which is defined as the sum of all weights of ASTZs in $\ASTZ_{n,l}^{i,j}$.

\begin{proposition}
	
	\begin{align}
	\label{propenum:1}
	&Z(n,l;i,j)=
	\begin{cases}
	0 & \text{if } i>j,\\
	R & \text{if } i=j.
	\end{cases}\\
	\label{propenum:2}
	&Z(n,l;i,j)=Z(n-1,l;i,j) \text{ if } j<n.\\
	\label{propenum:3}
	&Z(n,l;i,n)=Z(n-1,l+2;i-1,n-1) \text{ if } i>1.
	\end{align}
	
\end{proposition}

\begin{proof}
	Let $A \in \ASTZ_{n,l}^{i,j}$. Then the $i-1$ bottommost rows and the $n-j$ topmost rows have to be of the form $(0\: \dots\: 0\: 1)$. If $i>j$, then all rows are of the form $(0\: \dots\: 0\: 1)$ which contradicts the fact that there is a $1$-column in the left half of $A$. If $i=j$, then all rows but the $i\textsuperscript{th}$ one from the bottom are of the form $(0\: \dots\: 0\: 1)$, whereas the $i\textsuperscript{th}$ row is of the form $(1\: 0\: \dots\: 0)$. If $j<n$, removing the topmost row yields \eqref{propenum:2}, whereas removing the bottommost row if $i>1$ yields \eqref{propenum:3}.
\end{proof}

We denote the subset of all CSSPPs in $\CSSPP_{n,k}$ with only one row and exactly $j$ parts by  $\CSSPP_{n,k}^{j}$. Note that CSSPPs with only one row are ordinary integer partitions. It is obvious that $\CSSPP_{n,k}^{j}=\emptyset$ if $n < j$ and that $|\CSSPP_{n,k}^{j}|=|\CSSPP_{n-1,k}^{j}|$ if $n > j$. The last observation is a consequence of \eqref{propenum:2}.

The purpose of the present paper is to construct bijections
\begin{equation*}
\bigcup_{i=1}^{n} \ASTZ_{n,l}^{i,j} \longleftrightarrow \CSSPP_{n,l-1}^{j}
\end{equation*}
preserving the weights $(\mu,p,q)$ and $(\mu_d,p_d,q)$ for any $1 \le d \le l-1$, respectively. That is, we solve the following task:

\begin{problem}
	For $n \ge 1$, $l \ge 2$ and $1 \le d \le l-1$, construct bijections between $(n,l)$-ASTZs with exactly one $1$-column in the left half, $\mu$ entries equal to $-1$, $p$ $10$-columns in the left half, $q$ $10$-columns in the right half, and a $0$-column at position~$j$ and CSSPPs of class $l-1$ with exactly one row and $j$ parts, thereof $\mu$ $d$-special parts, $p$ parts at position~$k$ that are equal to $k-1+d$, and $q$ parts equal to $1$.
\end{problem}

As a side benefit, we obtain an enumeration formula for the number of elements in $\ASTZ_{n,l}^{i,j}$:

\begin{theorem}\label{thm:MainEnumeration}
	The number of alternating sign trapezoids $A \in \ASTZ_{n,l}^{i,j}$ such that $\mu(A)=\mu$, $p(A)=p$, and $q(A)=q$ is given by
	\begin{equation}\label{eq:MainEnumeration}
	\binom{j-i}{\mu+p+q} \binom{j+i+l-q-5}{\mu} - \binom{j-i-1}{\mu+p+q} \binom{j+i+l-q-4}{\mu}.
	\end{equation}
\end{theorem}

We proceed as follows: In Section~\ref{sec:LatticePathsRepresentations}, we discuss the lattice path representations of ASTZs and CSSPPs as our main tool. In Section~\ref{sec:Bijection}, we present the actual construction of the bijections. First, we start with a detailed description of how to go from ASTZs to CSSPPs in Section~\ref{sec:ASTZtoCSSPP}. Then, we briefly describe the inverse from CSSPPs to ASTZs in Section~\ref{sec:CSSPPtoASTZ} and discuss the case~$l=1$ in Section~\ref{sec:QAST}. At the end, we finish with some concluding remarks in Section~\ref{sec:Remarks}. These include brief comments on what we can say about the bijections at the level of partitions in Section~\ref{sec:PartitionLevel}, other statistics preserved by the bijections in Section~\ref{sec:Behrend}, some details on other established bijections in Section~\ref{sec:OtherBijections}, and a final outlook on obstacles to a generalisation to all ASTZs and CSSPPs in Section~\ref{sec:GeneralBijection}. Note that, for the sake of simplicity, we often write \emph{the bijection} despite meaning a \emph{family of bijections}.

\section{Lattice path representations}
\label{sec:LatticePathsRepresentations}

\subsection{ASTZs as lattice paths}
\label{sec:ASTZLatticePaths}

The bijections presented in this paper are based on the concept of osculating paths. \emph{Osculating paths} are lattice paths that neither cross nor share edges but potentially share points. The idea of describing ASMs in terms of osculating paths dates back to \cite{BH95} and was further investigated in \cite{Brak97}, \cite{Err01}, and \cite{Beh08}, amongst others. In the following, we present a generalisation of this model.

\begin{definition}

Let $\mathcal{T}_{n,l}$ be the trapezoidal region of the square grid that consists of $n$ centred rows of lengths~$l+2n-2$, $l+2n-4$, \dots, $l$ as pictured below.

	
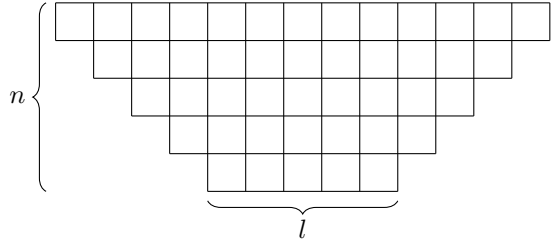
\begin{center}
	\begin{tikzpicture}[scale=0.25,baseline=(current bounding box.center)]
		
		\draw (-8,10) -- (18,10);
		\draw (-8,8) -- (18,8);
		\draw (-6,6) -- (16,6);
		\draw (-4,4) -- (14,4);
		\draw (-2,2) -- (12,2);
		\draw (0,0) -- (10,0);
		
		\draw (-8,10) -- (-8,8);
		\draw (-6,10) -- (-6,6);
		\draw (-4,10) -- (-4,4);
		\draw (-2,10) -- (-2,2);
		\draw (0,10) -- (0,0);
		\draw (2,10) -- (2,0);
		\draw (4,10) -- (4,0);
		\draw (6,10) -- (6,0);
		\draw (8,10) -- (8,0);
		\draw (10,10) -- (10,0);
		\draw (12,10) -- (12,2);
		\draw (14,10) -- (14,4);
		\draw (16,10) -- (16,6);
		\draw (18,10) -- (18,8);
		
		\draw [decorate,decoration={brace,amplitude=5pt,mirror}] (-8.5,10) -- (-8.5,0);
		\draw [decorate,decoration={brace,amplitude=5pt,mirror}] (0,-.5) -- (10,-.5);
		
		\node at (-10,5) {$n$};
		\node at (5,-2) {$l$};
		
	\end{tikzpicture}
\end{center}

We consider the following tilings of this region that naturally define a family of osculating paths: An \emph{$(n,l)$-osculating paths configuration} is a tiling of $\mathcal{T}_{n,l}$ with the six tiles
\begin{center}
	\begin{tikzpicture}[scale=0.25,baseline=(current bounding box.center)]
		
		\draw (0,0) -- (2,0) -- (2,2) -- (0,2) -- cycle;

	\end{tikzpicture}\qquad
	\begin{tikzpicture}[scale=0.25,baseline=(current bounding box.center)]
		
		\draw (0,0) -- (2,0) -- (2,2) -- (0,2) -- cycle;
		\draw[mycolor,ultra thick] (0,1) -- (2,1);
		
	\end{tikzpicture}\qquad
	\begin{tikzpicture}[scale=0.25,baseline=(current bounding box.center)]
		
		\draw (0,0) -- (2,0) -- (2,2) -- (0,2) -- cycle;
		\draw[mycolor,ultra thick] (1,0) -- (1,2);
		
	\end{tikzpicture}\qquad
	\begin{tikzpicture}[scale=0.25,baseline=(current bounding box.center)]
		
		\draw (0,0) -- (2,0) -- (2,2) -- (0,2) -- cycle;
		\draw[mycolor,ultra thick,rounded corners] (0,1) -- (1,1) -- (1,2);
		
	\end{tikzpicture}\qquad
	\begin{tikzpicture}[scale=0.25,baseline=(current bounding box.center)]
		
		\draw (0,0) -- (2,0) -- (2,2) -- (0,2) -- cycle;
		\draw[mycolor,ultra thick,rounded corners] (1,0) -- (1,1) -- (2,1);
		
	\end{tikzpicture}\qquad
	\begin{tikzpicture}[scale=0.25,baseline=(current bounding box.center)]
		
		\draw (0,0) -- (2,0) -- (2,2) -- (0,2) -- cycle;
		\draw[mycolor,ultra thick,rounded corners] (0,1) -- (1,1) -- (1,2);
		\draw[mycolor,ultra thick,rounded corners] (1,0) -- (1,1) -- (2,1);
		
	\end{tikzpicture}
\end{center}
such that
\begin{compactitem}
	\item there are at most $n$ paths;
	\item each path starts vertically at the left edge;
	\item each path ends horizontally at the right edge.
\end{compactitem}
\end{definition}

Figure~\ref{fig:exampleLatticePathASTZ} shows an example of an osculating paths configuration. It corresponds to the ASTZ in Figure~\ref{fig:exampleASTZ}; we address the bijective correspondence between ASTZs and osculating paths configuration in the proposition below. We admit that one might confuse the tileable regions $\mathcal{T}_{n,l}$ with shifted Young diagrams. In Section~\ref{sec:PartitionLevel}, we shall exploit this resemblance in order to interpret ASTZs with exactly one $1$-column in the left half as certain integer partitions. Note that if we replace the \emph{osculations}
\resizebox{!}{\baselineskip}{\begin{tikzpicture}[scale=0.25]
		
		\draw (0,0) -- (2,0) -- (2,2) -- (0,2) -- cycle;
		\draw[mycolor,ultra thick,rounded corners] (0,1) -- (1,1) -- (1,2);
		\draw[mycolor,ultra thick,rounded corners] (1,0) -- (1,1) -- (2,1);
		
\end{tikzpicture}}
by \emph{crossings}
\resizebox{!}{\baselineskip}{\begin{tikzpicture}[scale=0.25]
		
		\draw (0,0) -- (2,0) -- (2,2) -- (0,2) -- cycle;
		\draw[mycolor,ultra thick] (0,1) -- (2,1);
		\draw[mycolor,ultra thick] (1,0) -- (1,2);
		
\end{tikzpicture}}\,, we get tilings akin to \emph{bumpless pipe dreams} introduced by Lam, Lee, and Shimozono \cite{LLS21}, which are in bijective correspondence with alternating sign matrices. 

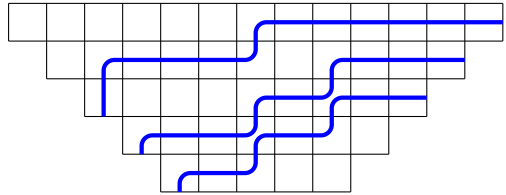
\begin{figure}[htb]
	\centering
	\begin{tikzpicture}[scale=0.25,baseline=(current bounding box.center)]
		
		\draw (-8,10) -- (18,10);
		\draw (-8,8) -- (18,8);
		\draw (-6,6) -- (16,6);
		\draw (-4,4) -- (14,4);
		\draw (-2,2) -- (12,2);
		\draw (0,0) -- (10,0);
		
		\draw (-8,10) -- (-8,8);
		\draw (-6,10) -- (-6,6);
		\draw (-4,10) -- (-4,4);
		\draw (-2,10) -- (-2,2);
		\draw (0,10) -- (0,0);
		\draw (2,10) -- (2,0);
		\draw (4,10) -- (4,0);
		\draw (6,10) -- (6,0);
		\draw (8,10) -- (8,0);
		\draw (10,10) -- (10,0);
		\draw (12,10) -- (12,2);
		\draw (14,10) -- (14,4);
		\draw (16,10) -- (16,6);
		\draw (18,10) -- (18,8);
		
		\draw[mycolor,ultra thick,rounded corners] (-3,4) -- (-3,7) -- (5,7) -- (5,9) -- (18,9);
		\draw[mycolor,ultra thick,rounded corners] (-1,2) -- (-1,3) -- (5,3) -- (5,5) -- (9,5) -- (9,7) -- (16,7);
		\draw[mycolor,ultra thick,rounded corners] (1,0) -- (1,1) -- (5,1) -- (5,3) -- (9,3) -- (9,5) -- (14,5);
		
	\end{tikzpicture}
	
	\caption{The $(5,5)$-osculating paths configuration that corresponds to the $(5,5)$-ASTZ in Figure~\ref{fig:exampleASTZ}.}
	\label{fig:exampleLatticePathASTZ}
	
\end{figure}

\begin{proposition}
	For $n \ge 1$ and $l \ge 2$, there is a bijective correspondence between $(n,l)$-ASTZs and $(n,l)$-osculating paths configurations which maps $A\in\ASTZ_{n,l}$ to an $(n,l)$-osculating paths configuration with $r(A)$ paths.
\end{proposition}

\begin{proof}
	It is easy to see how an ASTZ gives rise to an osculating paths configuration. Every entry of $A\in\ASTZ_{n,l}$ corresponds in an obvious bijection to a unit square in $\mathcal{T}_{n,l}$. For every $1$-column in the left half of $A$, we construct a path in $\mathcal{T}_{n,l}$ as follows: Each path starts vertically at the corresponding square of the bottommost entry of the $1$-column and moves upwards. If the path reaches a square corresponding to a $1$ or a $-1$, then it turns to the right or to the left, respectively.  If the path reaches a square that corresponds to the bottommost entry of a $10$-column, it also turns to the left. Thus, the path ends horizontally at a square that corresponds to the bottommost entry of a $0$-column in the right half of $A$. After all paths are determined, we have obtained a tiling of the region $\mathcal{T}_{n,l}$ with the tiles
	\begin{center}
		\begin{tikzpicture}[scale=0.25,baseline=(current bounding box.center)]
			
			\draw (0,0) -- (2,0) -- (2,2) -- (0,2) -- cycle;
			
		\end{tikzpicture}\qquad
		\begin{tikzpicture}[scale=0.25,baseline=(current bounding box.center)]
			
			\draw (0,0) -- (2,0) -- (2,2) -- (0,2) -- cycle;
			\draw[mycolor,ultra thick] (0,1) -- (2,1);
			
		\end{tikzpicture}\qquad
		\begin{tikzpicture}[scale=0.25,baseline=(current bounding box.center)]
			
			\draw (0,0) -- (2,0) -- (2,2) -- (0,2) -- cycle;
			\draw[mycolor,ultra thick] (1,0) -- (1,2);
			
		\end{tikzpicture}\qquad
		\begin{tikzpicture}[scale=0.25,baseline=(current bounding box.center)]
			
			\draw (0,0) -- (2,0) -- (2,2) -- (0,2) -- cycle;
			\draw[mycolor,ultra thick,rounded corners] (0,1) -- (1,1) -- (1,2);
			
		\end{tikzpicture}\qquad
		\begin{tikzpicture}[scale=0.25,baseline=(current bounding box.center)]
			
			\draw (0,0) -- (2,0) -- (2,2) -- (0,2) -- cycle;
			\draw[mycolor,ultra thick,rounded corners] (1,0) -- (1,1) -- (2,1);
			
		\end{tikzpicture}\qquad
		\begin{tikzpicture}[scale=0.25,baseline=(current bounding box.center)]
			
			\draw (0,0) -- (2,0) -- (2,2) -- (0,2) -- cycle;
			\draw[mycolor,ultra thick] (0,1) -- (2,1);
			\draw[mycolor,ultra thick] (1,0) -- (1,2);
			
		\end{tikzpicture}
	\end{center}
	We redirect the paths by replacing the crossings by osculations. In the end, $A$ is mapped to an $(n,l)$-osculating paths configuration where the number of paths equals $r(A)$.
	
	Conversely, given an $(n,l)$-osculating paths configuration, we next construct a corresponding $(n,l)$-ASTZ. First, we map the nonempty tiles along the paths to the entries $-1$, $0$, and $+1$ as follows:
	\begin{equation*}
		\def\arraystretch{2}
		\begin{array}{ccl}
	\begin{tikzpicture}[scale=0.25,baseline=(current bounding box.center)]
		\draw (0,0) -- (2,0) -- (2,2) -- (0,2) -- cycle;
		\draw[mycolor,ultra thick] (0,1) -- (2,1);
	\end{tikzpicture},\,
	\begin{tikzpicture}[scale=0.25,baseline=(current bounding box.center)]
		\draw (0,0) -- (2,0) -- (2,2) -- (0,2) -- cycle;
		\draw[mycolor,ultra thick] (1,0) -- (1,2);
	\end{tikzpicture},\,
	\begin{tikzpicture}[scale=0.25,baseline=(current bounding box.center)]
		\draw (0,0) -- (2,0) -- (2,2) -- (0,2) -- cycle;
		\draw[mycolor,ultra thick,rounded corners] (0,1) -- (1,1) -- (1,2);
		\draw[mycolor,ultra thick,rounded corners] (1,0) -- (1,1) -- (2,1);
	\end{tikzpicture}	& \longmapsto & 0 \\[2pt]
	\begin{tikzpicture}[scale=0.25,baseline=(current bounding box.center)]
		\draw (0,0) -- (2,0) -- (2,2) -- (0,2) -- cycle;
		\draw[mycolor,ultra thick,rounded corners] (1,0) -- (1,1) -- (2,1);
	\end{tikzpicture} & \longmapsto & +1\\
	\begin{tikzpicture}[scale=0.25,baseline=(current bounding box.center)]
		\draw (0,0) -- (2,0) -- (2,2) -- (0,2) -- cycle;
		\draw[mycolor,ultra thick,rounded corners] (0,1) -- (1,1) -- (1,2);
	\end{tikzpicture} & \longmapsto & \begin{cases*}
	0 & at the right edge\\
	-1 & otherwise\\
	\end{cases*}\\
		\end{array}
	\end{equation*}

Thus, we obtain a trapezoidal array $(a_{i,j})_{1 \le i \le n,i \le j \le 2n+l-1-i}$ with $-1$s, $0$s, $+1$s, and blank tiles~\resizebox{!}{\baselineskip}{\begin{tikzpicture}[scale=0.25]
		
		\draw (0,0) -- (2,0) -- (2,2) -- (0,2) -- cycle;
		
\end{tikzpicture}}\,. In this array, the nonzero integers alternate in sign since the paths in the osculating paths configuration go from the bottom left to the top right. It is also clear that, in each row and each column, the first nonzero entry from the left or the top, respectively, is $+1$. This observation suggests to call such an array a \emph{partial ASTZ}. We show that such a partial ASTZ can be completed to an ASTZ as follows: First, we map every blank tile at the right edge to $+1$. Finally, we map all the remaining blank tiles to $0$.

We have to show that the resulting array is indeed an ASTZ. After replacing the blank tiles, the rows and columns still alternate in sign. This is true since a blank tile is either mapped to $0$ or $1$, where the latter case only occurs at the right edge of the grid. If the row in the osculating paths configuration consisted of only blank tiles, then it is mapped to $(0\: \dots 0\: 1)$. Otherwise, it is mapped to $(*\: \dots\: *\: 0\: \dots\: 0\:1)$, where $*\: \dots\: *\:$ ends in $-1$ and alternates in sign by the previous observation. To see that also in columns the nonzero entries alternate in sign, we observe that above a blank tile at the right edge -- which is mapped to $1$ -- the first nonempty tile is either \resizebox{!}{\baselineskip}{\begin{tikzpicture}[scale=0.25]
		
		\draw (0,0) -- (2,0) -- (2,2) -- (0,2) -- cycle;
		\draw[mycolor,ultra thick] (0,1) -- (2,1);
		
\end{tikzpicture}} or \resizebox{!}{\baselineskip}{\begin{tikzpicture}[scale=0.25]

	\draw (0,0) -- (2,0) -- (2,2) -- (0,2) -- cycle;
	\draw[mycolor,ultra thick,rounded corners] (0,1) -- (1,1) -- (1,2);

\end{tikzpicture}}\,, which are mapped to $0$ or $-1$, respectively. Hence, the nonzero entries both in rows and columns alternate in sign.

We have already seen that each row and each column sums to $0$ or $1$. Regarding the rows, the row sum is clearly $1$ if the last entry in the row is equal to $1$. Otherwise, the last tile at the right edge of that row has to be \resizebox{!}{\baselineskip}{\begin{tikzpicture}[scale=0.25]
		
		\draw (0,0) -- (2,0) -- (2,2) -- (0,2) -- cycle;
		\draw[mycolor,ultra thick,rounded corners] (0,1) -- (1,1) -- (1,2);
		
\end{tikzpicture}}\,, which is mapped to $0$. But this tile has to be preceded by \resizebox{!}{\baselineskip}{\begin{tikzpicture}[scale=0.25]

	\draw (0,0) -- (4,0) -- (4,2) -- (0,2) -- cycle;
	\draw (2,0) -- (2,2);
	\draw (6,0) -- (8,0) -- (8,2) -- (6,2) -- cycle;
	\draw[mycolor,ultra thick,rounded corners] (1,0) -- (1,1) -- (4,1);
	\draw[mycolor,ultra thick,rounded corners] (6,1) -- (8,1);
	\node at (5,1) {$\cdots$};

\end{tikzpicture}}\, which is mapped to $1\:0\:\dots\: 0$. Thus, each row sums to $1$. Regarding the columns, we see that in each of the central $l-2$ columns every \resizebox{!}{\baselineskip}{\begin{tikzpicture}[scale=0.25]

	\draw (0,0) -- (2,0) -- (2,2) -- (0,2) -- cycle;
	\draw[mycolor,ultra thick,rounded corners] (1,0) -- (1,1) -- (2,1);

\end{tikzpicture}} has to be above a \resizebox{!}{\baselineskip}{\begin{tikzpicture}[scale=0.25]

	\draw (0,0) -- (2,0) -- (2,2) -- (0,2) -- cycle;
	\draw[mycolor,ultra thick,rounded corners] (0,1) -- (1,1) -- (1,2);

\end{tikzpicture}} since paths in the osculating paths configuration can only start at the left edge. Hence, the central $l-2$ columns sum to $0$ and the array is indeed an ASTZ. Finally, it is clear by construction that the number of paths is equal to the number of $1$-columns in the left half of the ASTZ.
\end{proof}

The ASTZs in $\ASTZ_{n,l}^{i,j}$ are thus mapped to osculating paths configurations with a single path. We interpret the squares of $\mathcal{T}_{n,l}$ as lattice points in the coordinate plane and identify the square corresponding to the bottommost entry of the column at position~$i$ with the origin. By this means, we obtain a one-to-one correspondence between $\ASTZ_{n,l}^{i,j}$ and lattice paths from $(-l-2i+3,0)$ to $(j-i,j-i)$ which only consist of rightward and upward unit steps and which do not cross the main diagonal~$y=x$.

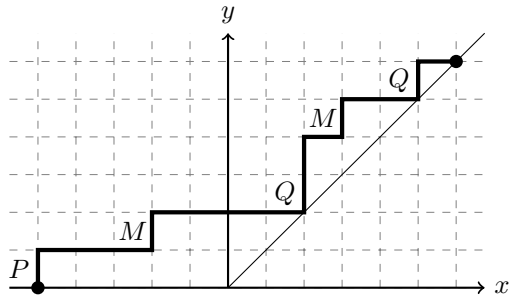
\begin{figure}[htb]
	\centering
	\begin{tikzpicture}[scale=.5]
	\draw [help lines,step=1cm,dashed] (-5.75,0) grid (6.75,6.75);
	
	\draw[->,thick] (-5.75,0)--(6.75,0) node[right]{$x$};
	\draw[->,thick] (0,0)--(0,6.75) node[above]{$y$};
	
	\draw (0,0)--(6.75,6.75);
	
	\fill (-5,0) circle (5pt);
	
	\fill (6,6) circle (5pt);
	
	\draw[ultra thick] (-5,0) -- (-5,1) -- (-2,1) -- (-2,2) -- (2,2) -- (2,4) -- (3,4) -- (3,5) -- (5,5) -- (5,6) -- (6,6);
	
	\node at (-5.5,0.5) {$P$};
	\node at (-2.5,1.5) {$M$};
	\node at (1.5,2.5) {$Q$};
	\node at (2.5,4.5) {$M$};
	\node at (4.5,5.5) {$Q$};
	
	
	\end{tikzpicture}
	
	\caption{The lattice path corresponding to $\bm{A}$.}
	\label{fig:LP1}
\end{figure}

The statistics on $A \in \ASTZ_{n,l}^{i,j}$ are reflected in the associated lattice path as follows: If the $1$-column we start with is a $10$-column, then the first step of the lattice path is an upward step; otherwise, it is a rightward step. For every $10$-column in the right half of $A$, we have a \emph{left turn} on the main diagonal, that is, a rightward step which is immediately followed by an upward step. All the other left turns in the lattice path result from the $-1$s in $A$. See Figure~\ref{fig:LP1} for an illustration of the weight of $\bm{A}$ in terms of its corresponding lattice path.

\subsection{CSSPPs as lattice paths}
\label{sec:CSSPPLatticePaths}

The representation of CSSPPs as a family of nonintersecting lattice paths is well-known; see Figure~\ref{fig:exampleLatticePathCSSPP} for an illustration.

\begin{figure}[ht]
	\centering
		\begin{ytableau}
			10 & 10 & 9 & 9 & 4 & 1\\
			\none & 8 & 8 & 8 & 2\\
			\none & \none & 7 & 6 & 1\\
			\none & \none & \none & 5\\
		\end{ytableau}
		$\quad\longleftrightarrow\quad$
		\begin{tikzpicture}[scale=.5,baseline=(current bounding box.center)]
		\draw [help lines,step=1cm,dashed] (0,0) grid (5.75,10.75);
		
		\draw[->,thick] (0,0)--(5.75,0) node[right]{$x$};
		\draw[->,thick] (0,0)--(0,10.75) node[above]{$y$};
		
		\fill (5,0) circle (5pt);
		\fill (3,0) circle (5pt);
		\fill (2,0) circle (5pt);
		\fill (0,0) circle (5pt);
		
		\fill (0,9) circle (5pt);
		\fill (0,7) circle (5pt);
		\fill (0,6) circle (5pt);
		\fill (0,4) circle (5pt);
		
		\draw[ultra thick] (0,9) --++ (1,0) --++ (0,-1) --++ (2,0) --++ (0,-5) --++ (1,0) --++ (0,-3) --++ (1,0);
		
		\draw[ultra thick] (0,7) --++ (2,0) --++ (0,-6) --++ (1,0) --++ (0,-1) ;
		
		\draw[ultra thick] (0,6) --++ (0,-1) --++ (1,0) --++ (0,-5) --++ (1,0);
		
		\draw[ultra thick] (0,4) --++ (0,-4) ;
		
		\node at (0.5,9.5) {$10$};
		\node at (1.5,8.5) {$9$};
		\node at (2.5,8.5) {$9$};
		\node at (3.5,3.5) {$4$};
		\node at (4.5,.5) {$1$};
		
		\node at (0.5,7.5) {$8$};
		\node at (1.5,7.5) {$8$};
		\node at (2.5,1.5) {$2$};
		
		\node at (0.5,5.5) {$6$};
		\node at (1.5,.5) {$1$};
		
		
		\end{tikzpicture}
		\caption{Example of the correspondence between a CSSPP and a family of nonintersecting lattice paths.}
		\label{fig:exampleLatticePathCSSPP}
\end{figure}
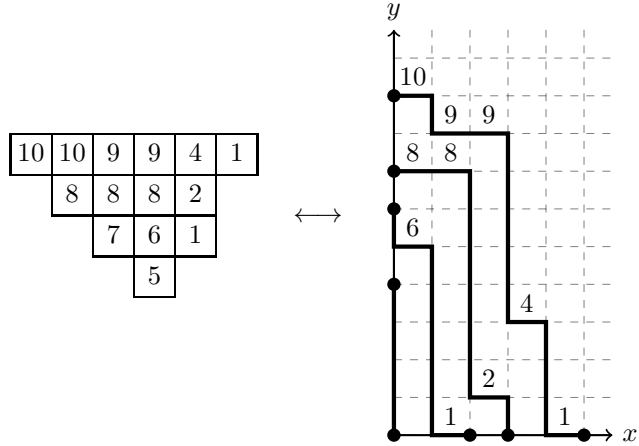

We encode a CSSPP row by row. If a row is given by the partition
\resizebox*{!}{\baselineskip}{\ytableaushort{{\lambda_1}{\lambda_2}{\dots}{\lambda_j}}}\,, it is associated to a lattice path from $(0,\lambda_1-1)$ to $(j-1,0)$ that only comprises rightward and downward unit steps in the following way: We start in the point $(-1,\lambda_1-1)$ and draw a lattice path to $(j-1,0)$ such that each part~$\lambda_m$ corresponds to a horizontal step at height~$\lambda_m-1$ for all $m\in\{1,\dots,j\}$; eventually, we delete the first step. In total, we obtain a family of nonintersecting lattice paths where the number of paths is given by the number of rows in the CSSPP.

A partition $\lambda \in \CSSPP_{n,k}^{j}$ is encoded by a single lattice path from $(0,j+k-1)$ to $(j-1,0)$. The statistics on $\lambda$ translate into the following properties: The lattice path has exactly $q(\lambda)$ horizontal steps at height $0$. The step right beneath the line $y=x+d$ is horizontal if and only if $p_d(\lambda)=1$. Finally, the number of horizontal steps below the line $y=x+k$ that are neither at height $0$ nor directly below $y=x+d$ equals $\mu_d(\lambda)$. For instance, the weight associated to the CSSPP  \resizebox*{!}{\baselineskip}{\ytableaushort{{11}9765411}} is illustrated in Figure~\ref{fig:RunningExampleCSSPP}. We call this CSSPP $\bm{\pi}$ for later reference.

\begin{figure}[htb]
	\centering
	\begin{tikzpicture}[scale=.5]
	\draw [help lines,step=1cm,dashed] (0,0) grid (7.75,10.75);
	
	\draw[->,thick] (0,0)--(7.75,0) node[right]{$x$};
	\draw[->,thick] (0,0)--(0,10.75) node[above]{$y$};
	
	\draw (0,3)--(5.5,8.5);
	\node at (7.25,8.5) {\contour{white}{$y=x+k$}};
	
	\draw (0,1)--(5.5,6.5);
	\node at (7.25,6.5) {\contour{white}{$y=x+d$}};
	
	\fill (7,0) circle (5pt);
	
	\fill (0,10) circle (5pt);

	\node at (2.5,5.5) {\contour{white}{$M$}};
	\node at (3.5,4.5) {\contour{white}{$P$}};
	\node at (4.5,3.5) {$M$};
	\node at (5.5,.5) {$Q$};
	\node at (6.5,.5) {$Q$};
	
	\draw[ultra thick] (0,10) --++ (0,-1) --++ (0,-1) --++ (1,0) --++ (0,-1) --++ (0,-1) --++ (1,0) --++ (0,-1) --++ (1,0) --++ (0,-1) --++ (1,0) --++ (0,-1) --++ (1,0) --++ (0,-1) --++ (0,-1) --++ (0,-1) --++ (1,0) --++ (1,0);
	
	\end{tikzpicture}
	
	\caption{The weighted lattice path associated to $\bm{\pi}$ for $d=1$. The class~$k$ of $\bm{\pi}$ is $k=3$.}
	\label{fig:RunningExampleCSSPP}
\end{figure}
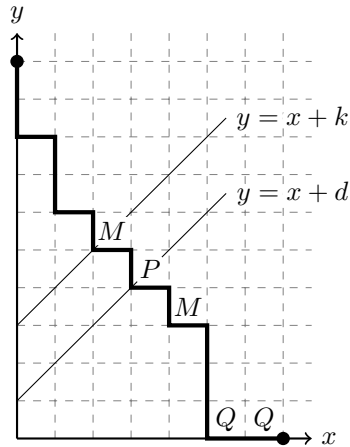

The number of partitions $\lambda \in \CSSPP_{n,k}^{j}$ with $\mu=\mu_d(\lambda)$, $p=p_d(\lambda)$, and $q=q(\lambda)$ for any $d\in\{1,\dots,k\}$ is given by
\begin{equation}\label{eq:EnumCSSPP}
\binom{j-1}{\mu+p+q} \binom{j-q+k-3}{\mu}.
\end{equation}

This product of two binomial coefficients can easily be seen as follows: Consider a partition of class~$k$ with $j$ parts. Its corresponding lattice path goes from $(0,j+k-1)$ to $(j-1,0)$ and crosses the line $y=x+k$. The intersection point divides the path into two smaller paths. The lower one contains all the steps possibly contributing a factor $M$, $P$, or $Q$ to the weight. It consists of $\mu+p+q$ horizontal steps, and, therefore, the coordinates of the intersection with the line $y=x+k$ are given by $(j-\mu-p-q-1,j+k-\mu-p-q-1)$. Next, we remove the following steps from the lower path: the $q$ horizontal steps at height~$0$, the vertical step from height~$0$ to height~$1$, and the step just below the line~$y=x+d$. Due to their uniquely determined positions, these steps can easily be reinserted after removing. In the end, we obtain a path consisting of $\mu$ horizontal and $j+k-\mu-q-3$ vertical steps. The final path looks like as shown in Figure~\ref{fig:CSSPPEnum}. These paths are enumerated by \eqref{eq:EnumCSSPP}. Note that we also obtain \eqref{eq:EnumCSSPP} by summing \eqref{eq:MainEnumeration} over all $1 \le i \le n$.

\begin{figure}[htb]
	\centering
	\begin{tikzpicture}[scale=.5]
	\draw [help lines,step=1cm,dashed] (-.75,.25) grid (4.75,10.75);

	\fill (0,10) circle (5pt);
	
	\fill (2,5) circle (5pt);
	
	\fill (4,1) circle (5pt);
	
	\draw[ultra thick] (0,10) --++ (0,-1) --++ (0,-1) --++ (1,0) --++ (0,-1) --++ (0,-1) --++ (1,0) --++ (0,-1) --++ (1,0) --++ (0,-1) --++ (0,-1) --++ (1,0) --++ (0,-1) --++ (0,-1);
	
	\draw[decorate,thick,decoration={brace,amplitude=.125cm}]  (0,10.25) -- (2,10.25);
	\node[anchor=south] at (1,10.5) {$j-\mu-p-q-1$};
	
	\draw[decorate,thick,decoration={brace,amplitude=.125cm}]  (-.25,5) -- (-.25,10);
	\node[anchor=east] at (-.75,7.5) {$\mu+p+q$};
	
	\draw[decorate,thick,decoration={brace,mirror,amplitude=.125cm}]  (2,.75) -- (4,.75);
	\node[anchor=north] at (3,.5) {$\mu$};
	
	\draw[decorate,thick,decoration={brace,mirror,amplitude=.125cm}]  (4.25,1) -- (4.25,5);
	\node[anchor=west] at (4.5,3) {$j+k-\mu-q-3$};
	
	\end{tikzpicture}
	
	\caption{The reduced lattice path representation of $\bm{\pi}$.}
	\label{fig:CSSPPEnum}
\end{figure}
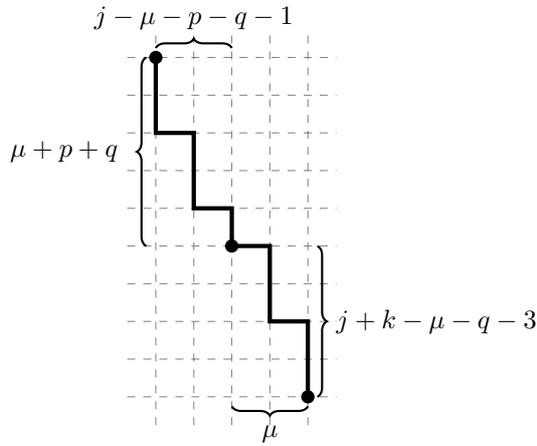

\section{The construction}
\label{sec:Bijection}

Before we discuss the actual construction of the bijection, we consider the case $i=j$. On the one hand, \eqref{propenum:1} implies that there is only one $A\in\ASTZ_{n,l}^{j,j}$ and that $\mu(A)=p(A)=q(A)=0$. On the other hand, there is also only one $\lambda\in\CSSPP_{n,l-1}^{j}$ such that $\mu(\lambda)=p(\lambda)=q(\lambda)=0$, namely the integer partition that consists of $j$ times the part $j+l-1$. This fact is easy to see by using the lattice path representation of CSSPPs. Thus, the case $i=j$ does not deserve further study.

For the sake of simplicity, we assume $i<j$ in the following construction. However, it is possible to accommodate the case $i=j$ by a few small but cumbersome adjustments.

\subsection{From ASTZs to CSSPPs}
\label{sec:ASTZtoCSSPP}

So far, we have seen how to represent elements from $\ASTZ_{n,l}^{i,j}$ and $\CSSPP_{n,l-1}^{j}$ as lattice paths. In this section, we construct a bijective and weight-preserving correspondence between these two families of lattice paths.

Consider $A \in \ASTZ_{n,l}^{i,j}$ and write $\mu = \mu(A)$, $p = p(A)$, and $q = q(A)$. We encode the respective lattice path by its \emph{left turn representation} \cite{KM93}, \cite{Kra97}: Since a lattice path with given starting point and endpoint is uniquely determined by the coordinates $(x_m, y_m)$ of its left turns, we can represent the lattice path associated to $A$ by the following two-rowed array:
\begin{equation}\label{eq:LeftTurn}
\begin{array}{ccccccccccc}
-2i-l+4 & \le & x_1 & < & x_2 & < & \dots & < & x_{\mu+q} & \le & j-i-1\\
p & \le & y_1 & < & y_2 & < & \dots & < & y_{\mu+q} & \le & j-i-1
\end{array}
\end{equation}
Since the lattice path stays weakly above the main diagonal, it follows that $x_m \le y_m$ for all $1 \le m \le \mu+q$.

Next, we transform each row of \eqref{eq:LeftTurn} into a lattice path that consists of upward and rightward unit steps. If the row is given by $a \le z_1 < \dots < z_c \le b$, then the lattice path goes from $(-c,a-b+c-1)$ to $(0,0)$ such that each $z_m$ corresponds to a horizontal step at height~$z_m-m-b+c$ for all $m \in \{1,\dots,c\}$.

We continue by possibly truncating these paths: Each factor of $Q$ in the weight of $A$ corresponds to a left turn with coordinates $(x_m,y_m)$ such that $x_m=y_m$. We remove these redundant pieces of information by deleting the corresponding horizontal steps in the path associated to the $x$-coordinates. If $p=0$, then $y_1=0$. In that case, we remove the first step in the path corresponding to the $y$-coordinates, too.

For instance, the lattice path in Figure~\ref{fig:LP1} corresponding to our running example $\bm{A}$ has the following left turn representation:

\begin{equation*}
\begin{array}{ccccccccccc}
-4 & \le & -2 & < & 2 & < & 3 & < & 5 & \le & 5\\
1 & \le & 1 & < & 2 & < & 4 & < & 5 & \le & 5\\
\end{array}
\end{equation*}

The corresponding reduced paths are as follows:

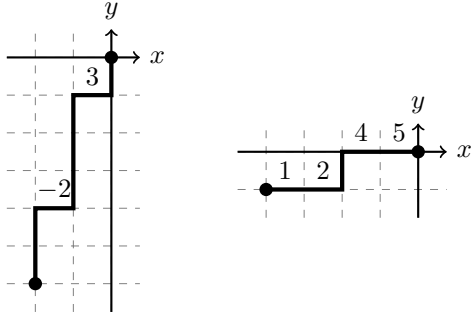
\begin{center}
	\begin{tikzpicture}[scale=.5,baseline=(current bounding box)]
	
	\draw [help lines,step=1cm,dashed] (-.75,-.75) grid (2.75,6.75);
	
	\draw[->,thick] (-.75,6)--(2.75,6) node[right]{$x$};
	\draw[->,thick] (2,-.75)--(2,6.75) node[above]{$y$};
	
	\fill (0,0) circle (5pt);
	\fill (2,6) circle (5pt);
	
	\draw[ultra thick] (0,0) -- ++(0,1) -- ++(0,1) -- ++(1,0) -- ++(0,1) -- ++(0,1) -- ++(0,1) -- ++(1,0) -- ++(0,1);
	
	\node at (.5,2.5) {$-2$};
	\node at (1.5,5.5) {$3$};
	
	\end{tikzpicture}
	\qquad
	\begin{tikzpicture}[scale=.5,baseline=(current bounding box)]
	
	\draw [help lines,step=1cm,dashed] (-.75,-.75) grid (4.75,1.75);
	
	\draw[->,thick] (-.75,1)--(4.75,1) node[right]{$x$};
	\draw[->,thick] (4,-.75)--(4,1.75) node[above]{$y$};
	
	\fill (0,0) circle (5pt);
	\fill (4,1) circle (5pt);
	
	\draw[ultra thick] (0,0) -- ++(1,0) -- ++(1,0) -- ++(0,1) -- ++(1,0) -- ++(1,0);
	
	\node at (.5,.5) {$1$};
	\node at (1.5,.5) {$2$};
	\node at (2.5,1.5) {$4$};
	\node at (3.5,1.5) {$5$};
	
	\end{tikzpicture}
\end{center}

We interpret these paths as nonintersecting lattice paths: The path associated to the $x$-coordinates has $S_x^i \coloneqq (-\mu,-j-i-l+4+\mu+q)$ as starting point and $(0,0)$ as endpoint; the other path is shifted by one unit step to the left such that it has $S_y^i \coloneqq (-\mu-p-q,-j+i+\mu+p+q)$ as starting point and $(-1,0)$ as endpoint. These paths are nonintersecting by construction.

Hence, the path with starting point $S_x^i$ ends with a vertical step from $(0,-1)$ to $(0,0)$. We remove this step and change the endpoint from $(0,0)$ to $E_x \coloneqq (0,-1)$. In addition, we change the endpoint of the other path from $(-1,0)$ to $E_y \coloneqq (0,0)$ by adding a horizontal step from $(-1,0)$ to $(0,0)$. In the case of our running example $\bm{A}$, we obtain the two nonintersecting lattice paths displayed in Figure~\ref{fig:RunningExampleLatticePaths}.

\begin{figure}[htb]
	\centering
	\begin{tikzpicture}[scale=.5]
	\draw [help lines,step=1cm,dashed] (-5.75,-6.75) grid (0.75,0.75);
	
	\fill (-5,-1) circle (5pt);
	
	\fill (0,0) circle (5pt);
	
	\draw[ultra thick] (0,0) node[right]{\contour{white}{$E_y$}} --++ (-1,0) --++ (-1,0) --++ (-1,0) --++ (0,-1) --++ (-1,0) --++ (-1,0) node[left]{\contour{white}{$S_y^i$}};
	
	\fill (-2,-6) circle (5pt);
	
	\fill (0,-1) circle (5pt);
	
	\draw[ultra thick] (0,-1) node[right]{\contour{white}{$E_x$}} --++ (-1,0) --++ (0,-1) --++ (0,-1) --++ (0,-1) --++ (-1,0) --++ (0,-1) --++ (0,-1) node[left]{\contour{white}{$S_x^i$}};
	\end{tikzpicture}
	
	\caption{$\bm{A}$ represented as two nonintersecting lattice paths.}
	\label{fig:RunningExampleLatticePaths}
\end{figure}
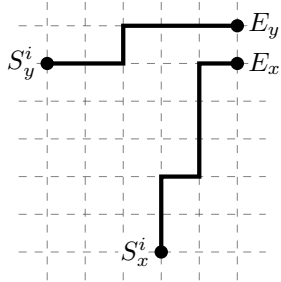

Let $\NILP(\{S_1,S_2\},\allowbreak \{E_1,E_2\})$ denote the set of pairs of nonintersecting lattice paths between the starting points $S_1$ and $S_2$ and the endpoints $E_1$ and $E_2$ which only consist of rightward and upward unit steps. The number of pairs of nonintersecting lattice paths in $\NILP(\{S_x^i,S_y^i\},\{E_x,E_y\})$ can be calculated with the help of the well-known Lindstr{\"o}m-Gessel-Viennot lemma and is given by \eqref{eq:MainEnumeration}. Thus, we have proved Theorem~\ref{thm:MainEnumeration}.

Next, we want show by a bijective proof that
\begin{multline}\label{eq:NILPs}
\NILP(\{S_x^{\tilde{\imath}},S_y^{\tilde{\imath}}\},\{E_x,E_y\}) \cup  \NILP(\{S_x^{\tilde{\imath}-1},S_y^{\tilde{\imath}-1}\},\{E_x,E_y\}) \\
= \NILP(\{S_x^{\tilde{\imath}},S_y^{\tilde{\imath}}-(0,1)\}, \{E_x-(0,1),E_y\}).
\end{multline}
Note that $S_x^{\tilde{\imath}-1}=S_x^{\tilde{\imath}}+(0,1)$ and $S_y^{\tilde{\imath}-1}=S_y^{\tilde{\imath}}-(0,1)$.
We interpret 
\begin{equation*}
\NILP(\{S_x^{\tilde{\imath}-1}, S_y^{\tilde{\imath}-1}\}, \{E_x,E_y\}) \subseteq \NILP(\{S_x^{\tilde{\imath}},S_y^{\tilde{\imath}}-(0,1)\}, \{E_x-(0,1),E_y\})
\end{equation*}
by shifting each path from $S_x^{\tilde{\imath}-1}$ to $E_x$ down by one unit step. The set
\begin{equation*}
\NILP(\{S_x^{\tilde{\imath}},S_y^{\tilde{\imath}}-(0,1)\}, \{E_x-(0,1),E_y\}) \setminus \NILP(\{S_x^{\tilde{\imath}-1},S_y^{\tilde{\imath}-1}\}, \{E_x,E_y\})
\end{equation*}
is the subset of $\NILP(\{S_x^{\tilde{\imath}},S_y^{\tilde{\imath}}-(0,1)\}, \{E_x-(0,1),E_y\})$ of nonintersecting lattice paths such that we create an intersection by shifting the path from $S_x^{\tilde{\imath}}$ to $E_x-(0,1)$ up by one unit step. We apply an instance of the Gessel-Viennot involution to these paths: Take the top right intersection point and switch the paths. Thus, we obtain a path from $S_x^{\tilde{\imath}}+(0,1)$ to $E_y$ and a path from $S_y^{\tilde{\imath}}-(0,1)$ to $E_x$. Shifting the former down by one unit step and the latter up by one unit step yield an element of $\NILP(\{S_x^{\tilde{\imath}}, S_y^{\tilde{\imath}}\},\{E_x,E_y\})$, which completes the proof of \eqref{eq:NILPs}.

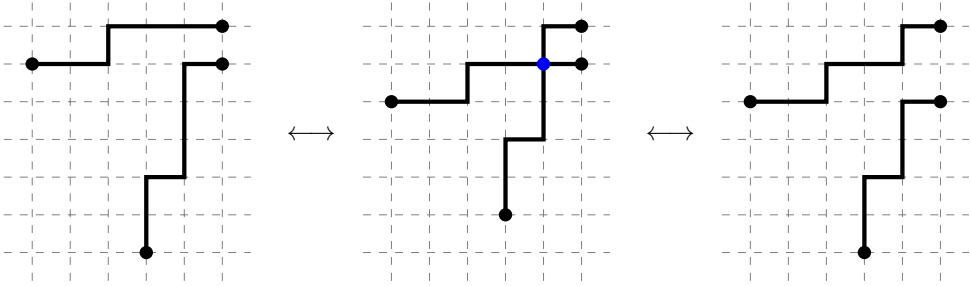
\begin{figure}[htb]
	\centering
	\begin{tikzpicture}[scale=.5,baseline=(current bounding box.center)]
	\draw [help lines,step=1cm,dashed] (-5.75,-6.75) grid (.75,.75);
	
	\fill (-5,-1) circle (5pt);
	
	\fill (0,0) circle (5pt);
	
	\draw[ultra thick] (0,0) --++ (-1,0) --++ (-1,0) --++ (-1,0) --++ (0,-1) --++ (-1,0) --++ (-1,0);
	
	\fill (-2,-6) circle (5pt);
	
	\fill (0,-1) circle (5pt);
	
	\draw[ultra thick] (0,-1) --++ (-1,0) --++ (0,-1) --++ (0,-1) --++ (0,-1) --++ (-1,0) --++ (0,-1) --++ (0,-1);
	\end{tikzpicture}
	\quad$\longleftrightarrow$\quad
	\begin{tikzpicture}[scale=.5,baseline=(current bounding box.center)]
	\draw [help lines,step=1cm,dashed] (-5.75,-6.75) grid (.75,.75);

	\fill (-5,-2) circle (5pt);
	
	\fill (0,0) circle (5pt);
	
	\draw[ultra thick] (0,-1) --++ (-1,0) --++ (-1,0) --++ (-1,0) --++ (0,-1) --++ (-1,0) --++ (-1,0);
	
	\fill (-2,-5) circle (5pt);
	
	\fill (0,-1) circle (5pt);
	
	\draw[ultra thick] (0,0) --++ (-1,0) --++ (0,-1) --++ (0,-1) --++ (0,-1) --++ (-1,0) --++ (0,-1) --++ (0,-1);
	
	\fill[mycolor] (-1,-1) circle (5pt);
	\end{tikzpicture}
	\quad$\longleftrightarrow$\quad
	\begin{tikzpicture}[scale=.5,baseline=(current bounding box.center)]
	\draw [help lines,step=1cm,dashed] (-5.75,-6.75) grid (.75,.75);
	
	\fill (-5,-2) circle (5pt);
	
	\fill (0,0) circle (5pt);
	
	\draw[ultra thick] (0,0) --++ (-1,0) --++ (0,-1) --++ (-1,0) --++ (-1,0) --++ (0,-1) --++ (-1,0) --++ (-1,0);
	
	\fill (-2,-6) circle (5pt);
	
	\fill (0,-2) circle (5pt);
	
	\draw[ultra thick] (0,-2) --++ (-1,0) --++ (0,-1) --++ (0,-1) --++ (-1,0) --++ (0,-1) --++ (0,-1);
	\end{tikzpicture}
	
	\caption{How to `shuffle' the two nonintersecting lattice paths corresponding to $\bm{A}$.}
	\label{fig:PathShiftingExample}
\end{figure}

For $A \in \ASTZ_{n,l}^{i,j}$, we `shuffle' the paths by repeating the process described above such that we shift and switch the paths $i-1$ times in total: At the beginning of the $\tilde{\imath}\textsuperscript{th}$ step, we have two nonintersecting lattice paths from $\{S_x^{i}, S_y^{i}-(0,\tilde{\imath}-1)\}$ to  $\{E_x-(0,\tilde{\imath}-1), E_y\}$; we shift the path with endpoint $E_y$ down by one unit step and the other path up by $(0,\tilde{\imath})$; after switching the paths at the top right intersection point, we shift the path with endpoint $E_x$ down by $(0,\tilde{\imath})$. This procedure yields an element of $\NILP(\{S_x^{i}, S_y^{i}-(0,\tilde{\imath})\}, \{E_x-(0,\tilde{\imath}), E_y\})$. After $i-1$ iterations, we obtain a pair of nonintersecting lattice paths from $\NILP(\{S_x^{i}, S_y^{i}-(0,i-1)\}, \{E_x-(0,i-1), E_y\})$. Figure~\ref{fig:PathShiftingExample} illustrates that process for our running example $\bm{A}$. Since $\bm{A}\in\ASTZ_{9,4}^{2,8}$, we have to shift and switch the paths once.

At the end, the upper path consists of $\mu+p+q$ horizontal steps and $j-\mu-p-q-1$ vertical steps, whereas the lower one consists of $\mu$ horizontal steps and $j+l-\mu-q-4$ vertical steps. Compare these paths with the paths in Figure~\ref{fig:CSSPPEnum}.

To construct the CSSPP, we have to rotate the upper path by $90^{\circ}$ and reflect the lower path along a horizontal or vertical axis. We have four possibilities in total for that. Each of them creates a possibly different but equally valid bijection. Finally, we connect the two paths and insert the missing steps that correspond to the steps we removed as described in Section~\ref{sec:CSSPPLatticePaths}. 

If we take the paths in Figure~\ref{fig:PathShiftingExample} and rotate the upper path clockwise by $90^{\circ}$, reflect the lower path along a vertical axis, and set $d=1$, our running example $\bm{A}$ is mapped to 
$\bm{\pi}$ as displayed in Figure~\ref{fig:RunningExampleCSSPP}. Table~\ref{tab:CompleteExample} provides a complete example for all CSSPPs in $\CSSPP_{5,1}^{5}$ with weight $M Q$, where we also set $d=1$ and take the same direction of rotation and the same axis of reflection as before.

\begin{table}[ht]
	\centering\scriptsize
	\caption{A complete example for the correspondence $\CSSPP_{5,1}^{5} \longleftrightarrow \bigcup_{i=1}^{5} \ASTZ_{5,2}^{i,5}$ with weight $M Q$ and $d=1$.}
		
	\begin{math}
	\begin{array}{ccc}
	\toprule
	
	\ytableaushort{66631}
	
	& 
	
	\ytableaushort{66621}
	
	& 
	
	\ytableaushort{66531}
	
	\\\midrule
	
	\nquad\begin{array}{cccccccccc}
	0  &  0  &  0  &  0  &  0  &  0  &  0  &  1  &  0  &  0\\
	&  0  &  0  &  0  &  0  &  1  &  0  & -1  &  1  &   \\
	&     &  0  &  0  &  0  &  0  &  0  &  1  &     &   \\
	&     &     &  0  &  0  &  0  &  1  &     &     &   \\
	&     &     &     &  1  &  0  &     &     &     &   
	\end{array}
	
	&
	
	\begin{array}{cccccccccc}
	0  &  0  &  0  &  0  &  0  &  0  &  1  &  0  &  0  &  0\\
	&  0  &  0  &  0  &  0  &  1  & -1  &  0  &  1  &   \\
	&     &  0  &  0  &  0  &  0  &  0  &  1  &     &   \\
	&     &     &  0  &  0  &  0  &  1  &     &     &   \\
	&     &     &     &  1  &  0  &     &     &     &   
	\end{array}
	
	&
	
	\begin{array}{cccccccccc}
	0  &  0  &  0  &  0  &  0  &  0  &  0  &  1  &  0  &  0\\
	&  0  &  0  &  0  &  0  &  0  &  1  & -1  &  1  &   \\
	&     &  0  &  0  &  0  &  0  &  0  &  1  &     &   \\
	&     &     &  1  &  0  &  0  &  0  &     &     &   \\
	&     &     &     &  0  &  1  &     &     &     &   
	\end{array}\nquad
	
	\\\midrule\midrule
	
	\ytableaushort{66521}
	& 
	
	\ytableaushort{66431}
	
	& 
	
	\ytableaushort{66421}
	
	\\\midrule
	
	\nquad\begin{array}{cccccccccc}
	0  &  0  &  0  &  0  &  0  &  0  &  1  &  0  &  0  &  0\\
	&  0  &  0  &  0  &  0  &  0  &  0  &  0  &  1  &   \\
	&     &  0  &  0  &  0  &  1  & -1  &  1  &     &   \\
	&     &     &  0  &  0  &  0  &  1  &     &     &   \\
	&     &     &     &  1  &  0  &     &     &     &   
	\end{array}
	
	&
	
	\begin{array}{cccccccccc}
	0  &  0  &  0  &  0  &  0  &  0  &  0  &  0  &  1  &  0\\
	&  0  &  0  &  0  &  0  &  1  &  0  &  0  &  0  &   \\
	&     &  0  &  0  &  0  &  0  &  0  &  1  &     &   \\
	&     &     &  1  &  0  & -1  &  1  &     &     &   \\
	&     &     &     &  0  &  1  &     &     &     &   
	\end{array}
	
	&
	
	\begin{array}{cccccccccc}
	0  &  0  &  0  &  0  &  0  &  0  &  0  &  0  &  1  &  0\\
	&  0  &  0  &  0  &  1  &  0  &  0  &  0  &  0  &   \\
	&     &  0  &  0  &  0  &  0  &  0  &  1  &     &   \\
	&     &     &  1  & -1  &  0  &  1  &     &     &   \\
	&     &     &     &  0  &  1  &     &     &     &   
	\end{array}\nquad
	
	\\\midrule\midrule
	
	\ytableaushort{65531}
	
	& 
	
	\ytableaushort{65521}
	
	& 
	
	\ytableaushort{65431}
	
	\\\midrule
	
	\nquad\begin{array}{cccccccccc}
	0  &  0  &  0  &  0  &  0  &  0  &  0  &  0  &  1  &  0\\
	&  0  &  0  &  0  &  0  &  0  &  1  &  0  &  0  &   \\
	&     &  1  &  0  &  0  &  0  & -1  &  1  &     &   \\
	&     &     &  0  &  0  &  0  &  1  &     &     &   \\
	&     &     &     &  0  &  1  &     &     &     &   
	\end{array}
	
	&
	
	\begin{array}{cccccccccc}
	0  &  0  &  0  &  0  &  0  &  0  &  0  &  0  &  1  &  0\\
	&  0  &  0  &  0  &  0  &  1  &  0  &  0  &  0  &   \\
	&     &  1  &  0  &  0  & -1  &  0  &  1  &     &   \\
	&     &     &  0  &  0  &  0  &  1  &     &     &   \\
	&     &     &     &  0  &  1  &     &     &     &   
	\end{array}
	
	&
	
	\begin{array}{cccccccccc}
	0  &  0  &  0  &  0  &  0  &  0  &  0  &  1  &  0  &  0\\
	&  0  &  0  &  0  &  0  &  0  &  0  &  0  &  1  &   \\
	&     &  0  &  0  &  0  &  1  &  0  &  0  &     &   \\
	&     &     &  1  &  0  & -1  &  1  &     &     &   \\
	&     &     &     &  0  &  1  &     &     &     &   
	\end{array}\nquad
	
	\\\midrule\midrule
	
	\ytableaushort{65421}
	
	& 
	
	\ytableaushort{64431}
	
	& 
	
	\ytableaushort{64421}
	
	\\\midrule
	
	\nquad\begin{array}{cccccccccc}
	0  &  0  &  0  &  0  &  0  &  0  &  0  &  1  &  0  &  0\\
	&  0  &  0  &  0  &  0  &  0  &  0  &  0  &  1  &   \\
	&     &  0  &  0  &  1  &  0  &  0  &  0  &     &   \\
	&     &     &  1  & -1  &  0  &  1  &     &     &   \\
	&     &     &     &  0  &  1  &     &     &     &   \\
	\end{array}
	
	&
	
	\begin{array}{cccccccccc}
	0  &  0  &  0  &  0  &  0  &  0  &  0  &  0  &  1  &  0\\
	&  0  &  0  &  0  &  1  &  0  &  0  &  0  &  0  &   \\
	&     &  1  &  0  & -1  &  0  &  0  &  1  &     &   \\
	&     &     &  0  &  0  &  0  &  1  &     &     &   \\
	&     &     &     &  0  &  1  &     &     &     &   \\
	\end{array}
	
	&
	
	\begin{array}{cccccccccc}
	0  &  0  &  0  &  0  &  0  &  0  &  0  &  0  &  1  &  0\\
	&  0  &  0  &  1  &  0  &  0  &  0  &  0  &  0  &   \\
	&     &  1  & -1  &  0  &  0  &  0  &  1  &     &   \\
	&     &     &  0  &  0  &  0  &  1  &     &     &   \\
	&     &     &     &  0  &  1  &     &     &     &   \\
	\end{array}\nquad
	
	\\\bottomrule
	\end{array}
	\end{math}
	\label{tab:CompleteExample}
\end{table}

In conclusion, we give a brief summary of the bijection $\bigcup_{i=1}^{n} \ASTZ_{n,l}^{i,j} \mapsto \CSSPP_{n,l-1}^{j}$. The bijection has three parameters that need to be fixed: $d\in\{1,\dots,l-1\}$, clockwise or counterclockwise rotation, and horizontal or vertical reflection. Consider $A\in\ASTZ_{n,l}^{i,j}$.
\begin{compactitem}
	\item Map $A$ to a lattice path from $(-l-2i+3,0)$ to $(j-i,j-i)$ that does not cross the main diagonal via the correspondence with osculating paths.
	\item Use the left turn representation to obtain paths from $S_x^{i}$ to  $E_x$ and from $S_y^{i}$ to  $E_y$.
	\item Repeat the following step $i-1$ times: Shift the path with endpoint $E_y$ down by one unit step and the path with endpoint $E_x$ up by one unit step and then switch the paths at the top right intersection point. 
	\item Rotate the upper path and reflect the lower path as determined beforehand. Join these two paths and insert the missing steps according to $d$, $p(A)$, and $q(A)$. Finally, recover the corresponding $\lambda \in \CSSPP_{n,l-1}^{j}$.
\end{compactitem}

\subsection{From CSSPPs to ASTZs}
\label{sec:CSSPPtoASTZ}

In this section, we succinctly describe the inverse mapping $\CSSPP_{n,k}^{j} \mapsto \bigcup_{i=1}^{n} \ASTZ_{n,k+1}^{i,j}$. We again fix the following three parameters: $d\in\{1,\dots,k\}$, clockwise or counterclockwise rotation, and horizontal or vertical reflection. Given $\lambda \in \CSSPP_{n,k}^{j}$, we analogously set $\mu=\mu_d(\lambda)$, $p=p_d(\lambda)$, and $q=q(\lambda)$ as before.

We interpret $\lambda$ as a lattice path, cut it into two parts along the line $y=x+k$ and truncate the lower path as described in Section~\ref{sec:CSSPPLatticePaths}. Next, we rotate the upper path by $90^{\circ}$ and we reflect the lower path according to the specified parameters. We place the paths as follows: the upper path goes from $(-\mu-p-q,-j+\mu+p+q+1)$ to $E_y = (0,0)$, the lower path goes from $(-\mu,-j-k+\mu+q+2)$ to $E_x=(0,-1)$.

As long as the two paths intersect, we repeat as follows: Switch the paths at the top right intersection point and shift the path with endpoint $E_y$ down by one unit step and the path with endpoint $E_x$ up by one unit step. Eventually, we obtain two nonintersecting lattice paths. We set $i$ to be the number of times we switched the paths increased by $1$. The lattice paths do not intersect anymore after at most $j-\mu-p-q-1$ times of shifting. Thus, $i$ is not larger than $j-\mu-p-q$, which implies $i \le j$. Since the final paths are nonintersecting, the upper path comprises a horizontal step from $(-1,0)$ to $(0,0)$. We remove this step and add a vertical step from $(0,-1)$ to $(0,0)$ to the lower path instead.

Next, we transform these two paths into a two-rowed array as presented in \eqref{eq:LeftTurn}. The lower path is mapped to $-2i-k+3 \le x_1 < \dots < x_\mu \le j-i-1$ and the upper path is mapped to $1 \le y_1 < \dots < x_{\mu+q} \le j-i-1$ such that $x_m$ and $y_m$ equal $h_m+m-\mu-q+j-i-1$ where $h_m$ denotes the height of the $m\textsuperscript{th}$ horizontal step from the left of the respective path.
Thus, we obtain a two-rowed array with $\mu$ entries in the first row and with $\mu+p+q-1$ entries in the second row. If $p=0$, we add $0$ as the first entry in the second row. Next, we have to insert $q$ additional entries in the first row. To add an additional entry, compare the two rows componentwise. Suppose the first row is $x_1 < \dots < x_m$ and the second one is $y_1 < \dots < y_{\mu+q}$. Search for the smallest integer~$a$ such that $x_a \ge y_a$. Then set $x_a \mapsto y_a$ and $x_b \mapsto x_{b-1}+1$ for all $a+1 < b \le m+1$. If there is no such integer $a$, then set $x_{m+1} \mapsto y_{m+1}$.

After having inserted $q$ entries, we finally obtain the left turn representation of a lattice path from which we can uniquely recover an element of $\ASTZ_{n,k+1}^{i,j}$. 

\subsection{The exceptional case $l=1$}
\label{sec:QAST}

The notion of ASTZs can be generalised to $l=1$. Despite being a natural generalisation, this case is more intricate and often needs special consideration. In the following, we define ASTZs for $l=1$ and present a refined enumeration but we omit a bijective proof.

\begin{definition}
	For $n \geq 1$, an \emph{$(n,1)$-alternating sign trapezoid} or \emph{quasi alternating sign triangle} (QAST) is an array $(a_{i,j})_{1 \le i \le n,i \le j \le 2n-i}$ of $-1$s, $0$s, and $+1$s
	such that the following three conditions hold:
	\begin{compactitem}
		\item the nonzero entries alternate in sign in each row and each column;
		\item the topmost nonzero entry in each column is $1$;
		\item the bottom row sums to $0$ or $1$ but every other row to $1$. 
	\end{compactitem}
\end{definition}

QASTs with $n$ rows are equinumerous with CSSPPs of class~$0$ with at most $n$ parts in the first row. For a refined enumeration, we define the following statistics.

Regarding QASTs, we define the statistics $\mu$ and $r$ similarly to the case $l \ge 2$, whereas we slightly adapt the statistics $p$ and $q$. Let $A\in\ASTZ_{n,1}$, then we define
\begin{align*}
p(A) &\coloneqq \text{\# $10$-columns among the $n-1$ leftmost columns of A,}\\
q(A) &\coloneqq \text{\# $10$-columns among the $n-1$ rightmost columns of A.}
\end{align*}
In order to define the weight $w(A)$, we introduce the \emph{Iverson bracket}: For a logical statement $S$, $\left[S\right]=1$ if $S$ holds true and $\left[S\right]=0$ otherwise. We now set $w(A)$ to be equal to
\begin{equation*}
M^{\mu(A)} R^{r(A)} P^{p(A)} Q^{q(A)} (P+Q-M)^{\left[\text{central column is a $10$-column}\right]}.
\end{equation*}
	
In the case of CSSPPs of class~$0$, we also adapt some of the statistics. Let $\pi\in\CSSPP_{n,0}$, then $r(\pi)$ counts the rows of $\pi$ as before. For the other statistics, we define
\begin{align*}
\mu(P) &\coloneqq \text{\# parts $\pi_{i,j}\in\{2,3,\dots,j-i+k\} \setminus \{j-i\}$,}\\
p(P) &\coloneqq \text{\# parts $\pi_{i,j}=j-i>1$,}\\
q(P) &\coloneqq \text{\# parts $\pi_{i,j}=1$ such that $j-i >1$.}
\end{align*}
We define the weight $w(\pi)$ to be
\begin{equation*}
M^{\mu(\pi)} R^{r(\pi)} P^{p(\pi)} Q^{q(\pi)} (P+Q-M)^{\left[\text{$\pi_{i,j}=1$ such that $j-i=1$}\right]}.
\end{equation*}

Fischer \cite{Fis19} proved that the sets $\{\pi \in \CSSPP_{n,0} \mid r(\pi) = r\}$ and $\{A \in \ASTZ_{n,1} \mid r(A) = r\}$ have the same cardinality for any $1 \le r \le n$. Moreover, the author \cite{Hon} showed that the generating functions for $\ASTZ_{n,1}$ and for $\CSSPP_{n,0}$ with respect to the aforementioned weights coincide and are given by \eqref{eq:GenFunc}. However, the sets $\{w(A) \mid A \in \ASTZ_{n,1}\}$ and $\{w(\pi) \mid \pi \in \CSSPP_{n,0}\}$ are not the same in general. By modifying the bijection we have for the case $l \ge 2$, we presume to obtain a bijective proof nonetheless. 

\section{Concluding remarks}
\label{sec:Remarks}

\subsection{The bijection at the level of partitions}
\label{sec:PartitionLevel}

Every single-row CSSPP is clearly a partition. But also every ASTZ with exactly one $1$-column in the left half can be interpreted as a partition. This can be seen as follows: Let $A\in\ASTZ_{n,l}^{i,j}$. The corresponding $(n,l)$-osculating paths configuration consists of two connected regions of blank tiles divided by a single path. There are two exceptions: either the path travels along the edge of the grid; or there is no path at all. In both cases, there is only one connected region of blank tiles. In any case, we interpret the (upper left) region of blank tiles as a shifted Young diagram and denote its shape by $\lambda(A)$. For instance, the partition corresponding to our running example $\bm{A}$ is $\lambda(\bm{A})=(20,16,13,11,10,5,1)$. Figure~\ref{fig:PartitionOfASTZ} illustrates how to read off the partition from an osculating paths configuration.

\begin{figure}[htb]
	\centering
	\begin{tikzpicture}[scale=0.25,baseline=(current bounding box.center)]
		
		\fill[gray] (-10,12) -- (14,12) -- (14,10) -- (8,10) -- (8,8) -- (4,8) -- (4,6) -- (2,6) -- (2,2) -- (0,2) -- (-2,2) -- (-2,4) -- (-4,4) -- (-4,6) -- (-6,6) -- (-6,8) -- (-8,8) -- (-8,10) -- (-10,10) -- cycle;
		
		\draw (-10,12) -- (14,12);
		\draw (-10,10) -- (14,10);
		\draw (-8,8) -- (12,8);
		\draw (-6,6) -- (10,6);
		\draw (-4,4) -- (8,4);
		\draw (-2,2) -- (6,2);
		\draw (0,0) -- (4,0);
		
		\draw (-10,12) -- (-10,10);
		\draw (-8,12) -- (-8,8);
		\draw (-6,12) -- (-6,6);
		\draw (-4,12) -- (-4,4);
		\draw (-2,12) -- (-2,2);
		\draw (0,12) -- (0,0);
		\draw (2,12) -- (2,0);
		\draw (4,12) -- (4,0);
		\draw (6,12) -- (6,2);
		\draw (8,12) -- (8,4);
		\draw (10,12) -- (10,6);
		\draw (12,12) -- (12,8);
		\draw (14,12) -- (14,10);

		\draw[mycolor,ultra thick,rounded corners] (1,0) -- (1,1) -- (3,1) -- (3,5) -- (5,5) -- (5,7) -- (9,7) -- (9,9) -- (12,9);
		
	\end{tikzpicture}
	$\quad\longmapsto\quad$
	(12,8,5,3,2)
	
	\caption{The upper left connected region of blank tiles in the osculating paths configuration interpreted as a shifted Young diagram.}
	\label{fig:PartitionOfASTZ}
	
\end{figure}
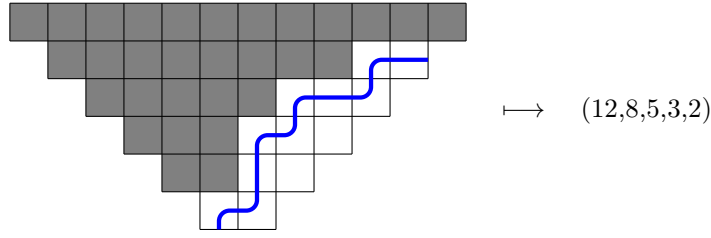

Thus, $\ASTZ_{n,l}^{i,j}$ corresponds to strict partitions $\lambda$ of length $n-i$ such that $\lambda_k=2n+l-2k$ for all $1 \le k \le n-j$. The counterparts of the statistics defined on $\ASTZ_{n,l}^{i,j}$ are 
\begin{align*}
	\tilde{\mu}(\lambda) &\coloneqq \text{\# parts $\lambda_{k}$ such that $\lambda_{k-1} + 2 \leq \lambda_k < 2n+l-2k-2$,}\\
	\tilde{p}(\lambda) &\coloneqq \text{\# $1$s in $\lambda$,}\\
	\tilde{q}(\lambda) &\coloneqq \text{\# parts $\lambda_{k}=2n+l-2k-2$,}
\end{align*}
where we set $\lambda_{n-i-1}=0$. Note that $\tilde{p}(\lambda)$ is either $0$ or $1$ since $\lambda$ is strict and therefore has only distinct parts.

When applying the bijection to $A\in\ASTZ_{n,l}^{i,j}$, we obtain two lattice paths by the left turn representation \eqref{eq:LeftTurn}, which we then repeatedly shuffle as described above for a total amount of $i-1$ times. After each of those steps, we could reassemble the two paths to a single one as we would do at the end to obtain an element of $\CSSPP_{n,l-1}^{j}$. In this way, we map $A$ to a sequence $(\lambda^{(0)},\lambda^{(1)},\dots,\lambda^{(i)})$ of partitions such that $\lambda^{(0)}=\lambda(A)$, and $\lambda^{(s+1)}$ is obtained after the $s$\textsuperscript{th} shuffling. We can show that, for any $1 \le s \le i$, it holds true that $\ell(\lambda^{(s)}) = j-i+s$ and $\lambda^{(s)}_1=j+i+l-1-s$. By interpreting $\lambda^{(s)}$ as a single-row CSSPP, we also see that the weights of all $\lambda^{(0)}, \lambda^{(1)}, \dots, \lambda^{(i)}$ coincide. In the end, $\lambda^{(i)}$ has length~$j$ and class~$l-1$ and is indeed an element of $\CSSPP_{n,l-1}^{j}$. However, due to the intricate construction of the bijection, we do not see how to fully characterise it at the level of partitions.

\subsection{Further statistics}
\label{sec:Behrend}

Let $M=(m_{i,j})_{1 \le i,j \le n}$ be an $n \times n$ ASM. Mills, Robbins, and Rumsey \cite{MRR83} defined the \emph{inversion number} of $M$ to be the sum of $m_{i,j} m_{\tilde{\imath},\tilde{\jmath}}$ over all $i$, $j$, $\tilde{\imath}$, and $\tilde{\jmath}$ such that $1 \le i < \tilde{\imath} \le n$ and $1 \le \tilde{\jmath} < j \le n$. This notion generalises the notion of the inversion number of permutation matrices. They conjectured that the number of $n \times n$ ASMs with inversion number $p$, $\mu$ entries equal to $-1$, and $m$ $0$s to the left of the 1 in the first row is equal to the number of DPPs without parts exceeding $n$ and with $p$ parts in total, thereof $\mu$ special parts and $m$ parts equal to $n$, which was proved by Behrend, Di Francesco and Zinn-Justin \cite{BFZ12}.

Behrend and Fischer \cite{BFPriv} conjecture a refined enumeration including an analogue of the inversion number for ASTZs: For $A\in\ASTZ_{n,l}$, let 
\begin{equation*}
\inv(A) \coloneqq \sum_{i < \tilde{\imath}} \sum_{j \le \tilde{\jmath}} a_{i,j} a_{\tilde{\imath},\tilde{\jmath}} + \text{\# $11$-columns among the $n$ leftmost columns of A}
\end{equation*}
denote the inversion number of $A$. For $\pi\in\CSSPP_{n,k}$, we define
\begin{equation*}
\inv(\pi) \coloneqq \text{\# parts $\pi_{i,j}>j-i+k$.}
\end{equation*}
Note that $\inv(\pi)$ equals the total number of parts minus $\mu_{d}(\pi)+p_d(\pi)+q(\pi)$ for any $d\in \{1,\dots,k\}$.

Then the quintuple $(\mu(A), r(A), p(A), q(A), \inv(A))$ on $A\in\ASTZ_{n,l}$ has conjecturally the same distribution as $(\mu_{d}(\pi), r(\pi), p_{d}(\pi), q(\pi), \inv(\pi))$ on $\pi\in\CSSPP_{n,l-1}$ if $d=l-1$. Computer experiments suggest that this equally holds true for $1 \le d < l-1$.

We see that the bijection presented in this paper respects these additional statistics since it can be shown that the inversion number of $A$ is equal to $j-\mu(A)-p(A)-q(A)$ for $A\in\ASTZ_{n,l}^{i,j}$.

\subsection{Comparison with other bijections}
\label{sec:OtherBijections}

So far, there are no complete bijections known between ASTZs and CSSPPs. Fischer \cite{Fis19} suggested that, for $1 \le d \le l-1$, there is a rather simple bijection between the sets $\{A \in \ASTZ_{n,l} \mid r(A) = 1, p(A)=p, q(A)=q\}$ and $\{\pi \in \CSSPP_{n,l-1} \mid r(\pi) = 1, p_d(\pi)=p, q(\pi)=q\}$. We present such a bijection which is based on the reflection principle. For the sake of brevity, we only describe how to map ASTZs to CSSPPs:

Given $A \in \ASTZ_{n,l}^{i,j}$ such that $p(A)=p$ and $q(A)=q$, consider the associated lattice path from $(-l-2i+3,0)$ to $(j-i,j-i)$ as described in Section~\ref{sec:ASTZLatticePaths}. Remove the first and the last step as well as all vertical steps that touch the main diagonal, which yields a lattice path from $(-l-2i-p+4,p+q)$ to $(j-i-1,j-i)$ that does not touch the main diagonal. Repeat the following step $i-1$ times: Shift the path to the right by two unit steps such that it intersects the main diagonal; then, reflect the path from the top right intersection point to the endpoint along the main diagonal. In the end, we have a lattice path from $(-l-p+2,p+q)$ to $(j-2,j-1)$. We rotate this path by $90^{\circ}$ -- for which we have two possibilities -- and insert the missing $q+2$ steps as described in Section~\ref{sec:CSSPPLatticePaths} to obtain a lattice path representation of a $\pi \in \CSSPP_{n,l-1}^{j}$ such that $p_d(\pi)=p$ and $q(\pi)=q$.

This bijection does not relate the number of $-1$s to the number of  $d$-special parts in contrast to the one presented in this paper. In order to incorporate the statistic~$\mu$, we required the detour of doubling the number of paths via the left turn representation and applying an instance of the Gessel-Viennot involution instead of dealing with a single path and the reflection principle. 

Even in the case of ASMs and DPPs, it is challenging to construct a bijection that includes the number of $-1$s and the number of special parts. The majority of established bijections only relate permutation matrices to DPPs without special parts. Ayyer \cite{Ayy10} provided a direct bijection that avoids the detour of nonintersecting lattice paths. Striker \cite{Str11} presented a bijection that also related the inversion number and the position of the $1$ in the last column of the permutation matrix to the number of parts and the number of maximal parts, respectively. Fulmek \cite{Ful20} established a bijection that preserves one additional statistic investigated by Behrend, Di Francesco and Zinn-Justin \cite{BFZ13}. The recent bijection by Fischer and Konvalinka \cite{FK20b}, however, is indeed a general bijection between all ASMs and DPPs but it does not preserve any of the statistics that we considered on CSSPPs.

With regard to ASMs and TSSCPPs, Striker \cite{Str18} gave an explicit weight-preserving bijection between permutation matrices and a subset of TSSCPPs. In addition, there are  two partial bijections between subclasses of ASMs and TSSCPPs to be mentioned here that do include ASMs with $-1$s: one independently found by Ayyer, Cori and Gouyou-Beauchamps \cite{ACG11} as well as by Striker \cite{Str11b} and another one by Biane and Cheballah \cite{CB12} which was later simplified by Bettinelli \cite{Bet16}.

\subsection{On a general bijection}
\label{sec:GeneralBijection}

The next step would be to extend the bijection presented in this paper to all ASTZs and CSSPPs. By relating the number of $-1$s to the number of $d$-special parts, we manage to incorporate a feature that other bijections lack. The key idea of our bijection is the correspondence between ASTZs and osculating paths. However, we are only able to deal with ASTZs whose osculating paths configurations consist of a single path. In this instance, the position of the rightmost $0$-column of the ASTZ determines the length of the corresponding single-row CSSPP.

Let us consider the naive approach towards a generalisation where we apply our bijection to each path in the osculating paths configuration separately. To this end, let $A\in\ASTZ_{n,l}$ such that $r(A) \ge 2$ and set $r=r(A)$. $A$ has exactly $r$ $0$-columns among the $n$ rightmost columns, say at the positions $(j_1,\dots,j_r)$. Thus, the naively generalised bijection would map $A$ to integer partitions of length~$j_1,\dots,j_r$ which we could arrange to a filled shifted Young diagram of shape~$(j_r,\dots,j_1)$. Two problems arise: On the one hand, we would have to check if these fillings are indeed strictly decreasing along columns. This might be accomplished but the major drawback is that, on the other hand, the number of $(n,l)$-ASTZs with $r$ $0$-columns at the positions $(j_1,\dots,j_r)$ and the number of CSSPPs of shape~$(j_r,\dots,j_1)$ and of class~$l-1$ are not the same in general.

For instance, there are five $(3,2)$-ASTZs with $0$-columns at positions~$1$ and $3$. Each of them corresponds to a family of two osculating paths. If we map each path individually to a partition, we would get a filled shifted Young diagram of shape~$(3,1)$. However, there are in total seven CSSPPs of shape~$(3,1)$ and of class~$1$. See Table~\ref{tab:Counterexample}. We would need an intermediate step to manipulate the osculating paths before using the existing bijection to map each path separately.  

\begin{table}[ht]
	\caption{There are five $(3,2)$-ASTZs  with $0$-columns at positions~$1$ and $3$ but seven CSSPPs  of shape~$(3,1)$ and of class~$1$.}
	\label{tab:Counterexample}
	\centering
	\begin{math}
	\begin{array}{c}
	\toprule
	\resizebox{\textwidth}{!}{\begin{math}
	\begin{array}{cccccc}
	1 & 0 & 0 & 0 & 0 & 0\\
	& 0 & 0 & 0 & 1 & \\
	& & 1 & 0 & & 
	\end{array}\qquad
	\begin{array}{cccccc}
	0 & 1 & 0 & 0 & 0 & 0\\
	& 0 & 0 & 0 & 1 & \\
	& & 1 & 0 & & 
	\end{array}\qquad
	\begin{array}{cccccc}
	0 & 0 & 1 & 0 & 0 & 0\\
	& 1 & -1 & 0 & 1 & \\
	& & 1 & 0 & & 
	\end{array}\qquad
	\begin{array}{cccccc}
	0 & 0 & 0 & 1 & 0 & 0\\
	& 1 & 0 & -1 & 1 & \\
	& & 1 & 0 & & 
	\end{array}\qquad
	\begin{array}{cccccc}
	0 & 0 & 0 & 0 & 1 & 0\\
	& 1 & 0 & 0 & 0 & \\
	& & 1 & 0 & & 
	\end{array}
	\end{math}}
	\\\midrule
	\resizebox{.75\textwidth}{!}{\begin{math}
	\begin{ytableau}
	4 & 4 & 4 \\
	\none & 2
	\end{ytableau}\qquad
	\begin{ytableau}
	4 & 4 & 3 \\
	\none & 2
	\end{ytableau}\qquad
	\begin{ytableau}
	4 & 4 & 2 \\
	\none & 2
	\end{ytableau}\qquad
	\begin{ytableau}
	4 & 4 & 1 \\
	\none & 2
	\end{ytableau}\qquad
	\begin{ytableau}
	4 & 3 & 3 \\
	\none & 2
	\end{ytableau}\qquad
	\begin{ytableau}
	4 & 3 & 2 \\
	\none & 2
	\end{ytableau}\qquad
	\begin{ytableau}
	4 & 3 & 1 \\
	\none & 2
	\end{ytableau}
	\end{math}}
	\\\bottomrule
	\end{array}
	\end{math}
\end{table}

\section*{Acknowledgements}

I thank Roger Behrend for helpful comments and for sharing his results discussed in Section~\ref{sec:Behrend}. Furthermore, I thank the anonymous referees for the feedback and their helpful suggestions that improved the paper.


\bibliography{RefBijections}
\bibliographystyle{alpha}

\textsc{Universit{\"a}t Wien, Fakult{\"a}t f{\"u}r Mathematik, Oskar-Morgenstern-Platz 1, 1090 Wien, Austria}

\textit{E-mail address:} \href{mailto:hans.hoengesberg@univie.ac.at}{hans.hoengesberg@univie.ac.at}

\end{document}